\documentclass[12pt]{amsart}
\usepackage{a4wide}
\usepackage[utf8]{inputenc}
\usepackage{tikz}
\usepackage[backend=biber,style=alphabetic,maxbibnames=99]{biblatex}  
\usepackage{multirow}
\usepackage{textcomp}
\usepackage{csquotes} 

\usepackage{marginnote}
\usepackage{xcolor}
\usepackage{listings}
\usepackage{hyperref}
\hypersetup{colorlinks=false}
\usepackage{amssymb}
\usepackage{enumitem}
\usepackage[none]{hyphenat}
\usepackage{bbm}
\usepackage{amsmath,amssymb}
\newcommand{\vertiii}[1]{{\left\vert\kern-0.25ex\left\vert\kern-0.25ex\left\vert #1 
    \right\vert\kern-0.25ex\right\vert\kern-0.25ex\right\vert}}

\newtheorem{theorem}{Theorem}[section]
\newtheorem{thmx}{Theorem}
 
\newtheorem{corollary}[theorem]{Corollary}
\newtheorem{proposition}[theorem]{Proposition}
\newtheorem{lemma}[theorem]{Lemma}
\newtheorem{claim}{Claim}
\newtheorem{question}{Question}

\theoremstyle{definition}
\newtheorem{definition}[theorem]{Definition}

\newtheorem{remark}[theorem]{Remark}

\newcommand{\R}{\mathbb{R}}

\newcommand{\N}{\mathbb{N}}

\newcommand{\Q}{\mathbb{Q}}

\newcommand{\Ke}{K_\varepsilon}

\DeclareMathOperator{\diam}{diam}
\DeclareMathOperator{\Lip}{Lip}
\DeclareMathOperator{\coLip}{co-Lip}

\definecolor{egraf}{rgb}{0.2,0.4,0}
\author{Jaan Kristjan Kaasik}\address[J. K. Kaasik]{Institute of Mathematics and Statistics, University of Tartu, Narva mnt 18, 51009 Tartu, Estonia
\newline
\href{https://orcid.org/0000-0001-6561-5557}{ORCID: \texttt{0000-0001-6561-5557} } 
\newline 
}
\author{Andr\'es Quilis}
\address[A. Quilis]{Instituto Universitario de Matemática Pura y Aplicada, Universitat Politècnica de València, Camí de Vera s/n, 46022, Valencia, Spain
\newline
\href{https://orcid.org/0000-0001-6022-9286}{ORCID: \texttt{0000-0001-6022-9286} } }
\email{\texttt{jaan.kristjan.kaasik@ut.ee}}
\email{\texttt{anquisan@upv.es}}

\subjclass[2020]{}
\title[Universality of Lipschitz quotients and the curve-flat index]{Universality of Lipschitz quotients and the curve-flat index}
\begin{document}

\begin{abstract}
  We study universality of Lipschitz quotients. First, we modify a construction of Johnson, Lindenstrauss, Preiss and Schechtman to obtain a complete separable metric space that has every complete separable metric space as a Lipschitz quotient. 
  
  Our main result is in the compact setting, where we prove that no such universal metric space can exist. We deduce this impossibility result by studying the curve-flat index, an ordinal index which provides a measure of the complexity of the curve-fragment structure in a metric space. We show that Lipschitz quotients cannot increase this index in compact domains; while there exist compact spaces with arbitrarily high countable curve-flat index. The main technical part of the paper is dedicated to proving a strong version of the latter fact: for every ordinal $\alpha$ and every compact metric space $M$, there exists a compact metric space $N$ such that the curve-flat quotient of $N$ of order $\alpha$ is almost-isometric to $M$.

\end{abstract}

\maketitle

\section{Introduction}

The \emph{co-Lipschitz} constant of a map $f\colon M\rightarrow N$, denoted by $\text{co-Lip}(f)$, is the infimum over all constants $C>0$ such that
\begin{equation}
B_N(f(x),r)\subset f(B_M(x,Cr))
\end{equation}
for all $x\in M$ and all $r>0$\footnote{$B_M(x,r)$ denotes the closed ball in $M$ centered at $x$ and of radius $r$}. Co-Lipschitz mappings (that is, mappings with finite co-Lipschitz constant) can be thought of as \emph{Lipschitz-open} mappings, a term by which they are also known. A \emph{$C$-Lipschitz quotient} is a Lipschitz map $f\colon M\rightarrow N$ between metric spaces which is also \emph{co-Lipschitz}, and that satisfies $\text{Lip}(f)\cdot\text{co-Lip}(f)\leq C$. 

Lipschitz quotients have been considered in the literature several times under different names (for instance, in \cite{Gro99}), but the first systematic study was done by Bates, Johnson, Lindenstrauss, Preiss and Schechtman in \cite{BJLPS99}, in the context of nonlinear maps between Banach spaces. The term \emph{Lipschitz quotient} derives its name precisely from the linear theory, as in Banach spaces the Open Mapping Theorem implies that linear quotients between Banach spaces are in particular Lipschitz quotients. 

Since then, several metric and Banach space properties have been shown to be preserved by Lipschitz (and other nonlinear) quotients, providing tools to prove the non-existence of such mappings between certain spaces. For instance, improving a previous result of \cite{BJLPS99}, Mendel and Naor showed in \cite{MN13} that $L_q$ is not a Lipschitz quotient of any subset of $L_p$ for $2\leq p<q$, by showing that Markov $p$-convexity is preserved through these mappings (see also \cite{NPS18}). Other notions known to enjoy this property are, for instance, the Rolewicz property $(\beta)$ (see \cite{DKR16} or \cite{BZ16}) or umbel $p$-convexity (see \cite{BG24}).

A foundational result in this regard is the following universality result (see Section \ref{Prelim} for precise definitions). 
\begin{theorem}[\cite{JLPS02}]
\label{Theorem_JLPS02}
    Every separable complete geodesic metric space is a $1$-Lipschitz quotient of any Banach space containing $\ell_1$.
\end{theorem}

\medskip

Another line of research concerns the study of Lipschitz quotients in finite-dimensional euclidean spaces. In this context, notions regarding connectedness and rectifiability in metric spaces play a more prominent role. Recall that a metric space is \emph{purely $1$-unrectifiable} if it does not contain any bi-Lipschitz copy of a positive measure subset of $\R$. In \cite{KMM06} Kun, Maleva and Máthé characterized compact purely $1$-unrectifiable subsets of $\R^n$ using a close relative to Lipschitz quotients (we refer again to Section \ref{Prelim} for the relevant definitions). A non-exhaustive sample of articles in this line is \cite{JLPS00,Cso01,MV21,HM23}. 

\medskip

In this paper, we study universality for Lipschitz quotients in large classes of metric spaces. Given a family $\mathcal{F}$ of metric spaces, we say that a metric space $X$ is \emph{universal for ($C$-)Lipschitz quotients onto $\mathcal{F}$ if there exists a ($C$)-Lipschitz quotient from $X$ onto any space in $\mathcal{F}$.} First, we consider the question of whether Theorem \ref{Theorem_JLPS02} can be extended to construct a separable complete metric space universal for $1$-Lipschitz quotients onto all complete separable metric spaces. Thus we arrive at our first main result.

\begin{thmx}
    \label{Main_Theorem_1_SUT}
    There exists a separable complete metric space $SUT$ which is universal for 1-Lipschitz quotients onto all separable complete metric spaces.
\end{thmx}

Since (Lipschitz) curves are preserved by Lipschitz mappings, such a metric space cannot be rectifiably connected. Indeed, we construct $SUT$ as a purely $1$-unrectifiable version of the main construction in \cite{JLPS02}, by iteratively gluing discrete points at rational distances in a metric tree-like fashion.

\medskip

However, the situation is different and more involved in the compact setting. This is reflected in the next theorem, which is our main contribution and to which most of this paper is dedicated to proving.

\begin{thmx}
\label{Main_Theorem_2_nocompactSUT}
    There is no compact metric space universal for Lipschitz quotients onto all compact metric spaces.
\end{thmx}

Let us point out intuitively why compactness yields a much different theory. If $f\colon M\rightarrow N$ is a Lipschitz quotient and $\gamma$ is a Lipschitz curve in $N$, then the co-Lipschitz condition lets us build finite sequences in $M$ that approximate a bi-Lipschitz copy of $\gamma$ in $M$. If, moreover, $M$ is compact, these finite sequences can be made to converge to a ``true" Lipschitz curve in $M$. This means that, in some sense, if $N$ is a Lipschitz quotient of a compact space $M$, then $M$ has a comparable structure of Lipschitz curves to that in $N$. This idea, that is already present in the proof of Theorem 2.1 in \cite{KMM06}, can be refined and exploited to prove the following general result.

\begin{thmx}
\label{Main_theorem_p1u}
    A Lipschitz quotient of a compact purely $1$-unrectifiable metric space is purely $1$-unrectifiable. 
\end{thmx}

In fact, Theorem \ref{Main_theorem_p1u} is a corollary of a more general fact (stated below), which states that the complexity of the curve-flat structure of a compact metric space cannot increase under Lipschitz quotients; which is one of the key ingredients for proving Theorem \ref{Main_Theorem_2_nocompactSUT}. In order to clarify the previous statement, we need to recall several notions that play a central role in this paper. 

A \emph{curve-fragment} is a Lipschitz map $\varphi\colon K\rightarrow M$ from a compact subset of the real line into a metric space. Given a metric space $(M,d)$, the \emph{curve-flat pseudometric $d_{cf}$} is the function $d_{cf}\colon M\times M\rightarrow[0,+\infty)$ given by
\begin{equation}
    d_{cf}(x,y):=\inf(\lambda([\min{K},\max{K}]\setminus K)),
\end{equation}
where the infimum is taken over all $1$-Lipschitz curve-fragments $\varphi\colon K\rightarrow M$ satisfying $\varphi(\min{K})=x$ and $\varphi(\max{K})=y$, and $\lambda$ denotes the Lebesgue measure.

Identifying points $x,y$ such that $d_{cf}(x,y)=0$, the pseudometric $d_{cf}$ defines a quotient set $M_{cf}:=M/d_{cf}$. This yields a metric space $(M_{cf},d_{cf})$ that encodes some of the curve-fragment structure of the original metric space. Moreover, this process can be iterated, and one can obtain a transfinite sequence of metric spaces
\begin{equation}
    M=M_{cf}^0\rightarrow M_{cf}^1\rightarrow\dots\rightarrow M_{cf}^\alpha\rightarrow \dots
\end{equation}
where $M_{cf}^{\alpha+1}$ is the curve-flat quotient of $M_{cf}^\alpha$ (limit ordinals are dealt with separately, see Section \ref{Prelim}). It is shown in \cite{AGPP22} (where the curve-flat pseudometric was introduced) that this sequence must stabilize at some ordinal $\alpha_{cf}(M)$ of no larger cardinality than the density character of $M$, which we call the \emph{curve-flat index of $M$}. This ordinal index is the measure of complexity of curve-fragment structure that we use, and we show the following general version of Theorem \ref{Main_theorem_p1u}.

\begin{thmx}
\label{Main_Theorem_index_through_quotients}
    Let $M$ and $N$ be compact metric spaces such that $N$ is a Lipschitz quotient of $M$. Then $\alpha_{cf}(N)\leq\alpha_{cf}(M)$.
\end{thmx}
Note that $\alpha_{cf}(M)=0$ if and only if $M$ is purely $1$-unrectifiable, and so indeed Theorem \ref{Main_theorem_p1u} is just a particular case of Theorem \ref{Main_Theorem_index_through_quotients}.

\medskip

Since the curve-flat index of a compact metric space is countable, the road map to proving Theorem \ref{Main_Theorem_2_nocompactSUT}, that relies on Theorem \ref{Main_Theorem_index_through_quotients}, becomes clear: it suffices to construct compact metric spaces with arbitrarily high (countable) curve-flat index. This follows from our next main result, which is a considerably stronger statement.

\begin{thmx}
    \label{Main_Theorem_4_compacthighordercf}
    For every compact metric space $M$, every countable ordinal $\alpha$ and every $\varepsilon>0$, there exists a compact metric space $N$ with $\alpha_{cf}(N)\geq \alpha$ such that $N_{cf}^{\alpha}$ is $(1+\varepsilon)$-isometric to $M$.
\end{thmx}

Previously, there was no explicit construction of a metric space with curve-flat index higher than $2$. The only example of a compact space $M$ with $\alpha_{cf}(M)=2$ was provided in \cite[Example 5.25]{AGPP22}, with a space satisfying that $M^1_{cf}$ is isometric to $[0,1]$ (see \cite[Remark 6.9]{Flo+24} for a proof of the latter fact). The authors of \cite{AGPP22} express their belief that the same idea could be iterated to produce compact spaces with arbitrarily high countable curve-flat index, but our construction relies on a different method which can use any compact metric space as a starting point. 

Before going into this construction, we take a brief detour to prove a first order version of Theorem \ref{Main_Theorem_4_compacthighordercf} without the compactness assumption but which always yields an exact isometry. 

\begin{thmx}
     \label{Main_Theorem_3_finiteorder_CF}
    For every complete metric space $M$ there exists a complete non-purely $1$-unrectifiable metric space $N$ with the same density character such that $N_{cf}$ is isometric to $M$.
\end{thmx}

We find it helpful to start here, mainly for two reasons: in the first place, the first order result contains the core ideas used in the general compact case, which is considerably more intricate. On the other hand, Theorem \ref{Main_Theorem_3_finiteorder_CF} is of independent interest in the study of curve-flat quotients, along the lines of the recent publications \cite{AGPP22,Flo+24}. Indeed, in \cite[Remark 6.9]{Flo+24} it was asked whether every (compact) metric space can be realized as the curve-flat quotient of another (compact) metric space. The combination of Theorems \ref{Main_Theorem_4_compacthighordercf} and \ref{Main_Theorem_3_finiteorder_CF} provides an almost complete answer to this question.

Theorems \ref{Main_Theorem_4_compacthighordercf} and \ref{Main_Theorem_3_finiteorder_CF} are proved using a new construction that we call the \emph{gapped graph} of a metric space. This is partly inspired by the \emph{cobweb} of a metric space, as defined in \cite[Section 6]{BanVovWoj11}, though there are several key differences in our approach. In our case, we start with a snowflake-like distortion of the metric space $M$, to which we attach segments with gaps (called \emph{gapped segments}) that encode the original metric. If done suitably, this increases the complexity of the curve-fragment structure while preserving properties like density character or compactness. In the high-order construction we attach inductively defined compact metric spaces with high curve-flat index, instead of gapped segments; this allows us to increase the curve-flat index of the resulting metric space to any given countable ordinal. 

The paper is organized as follows: Section \ref{Prelim} contains preliminary definitions and results used throughout the article, including notions related to curve-flat functions and Lipschitz quotients. In Section \ref{Section:Univ_Lips_Quotients_Sep} we construct the space SUT and prove Theorem \ref{Main_Theorem_1_SUT}. In Section \ref{Section:Lipschitz_quotients} we focus on the compact case, showing that the curve-flat index is preserved through Lipschitz quotients, leading to the proof of Theorem \ref{Main_Theorem_index_through_quotients}. In this section we also show how Theorems \ref{Main_Theorem_index_through_quotients} and \ref{Main_Theorem_4_compacthighordercf} (proven later) are sufficient to prove Theorem \ref{Main_Theorem_2_nocompactSUT}. In Section \ref{Section:Finite_order} we show Theorem \ref{Main_Theorem_3_finiteorder_CF}, and in Section \ref{Section:High_order} we modify the construction of the previous section to deal with the high order curve-flat quotients in the compact setting, thus obtaining Theorem \ref{Main_Theorem_4_compacthighordercf}. We conclude the paper with some open questions.

\section{Preliminaries}\label{Prelim}
Let $(M,d)$ and $(N,\rho)$ be metric spaces. For $A\subset M$, $x\in M$ and $r>0$, we use the notation
\begin{equation}B(x,r)=\{y\in M:d(x,y)\leq r\};\end{equation}
\begin{equation}\diam(A)=\sup\{d(x,y):x,y\in A\}.\end{equation}
\smallskip
Given a Lipschitz function $f\colon (M,d)\rightarrow (N,\rho)$, we denote its optimal Lipschitz constant by $\Lip(f)$. We say that a bijection $\varphi\colon (M,d)\rightarrow (N,\rho)$ between metric spaces $(M,d)$ and $(N,\rho)$ is a \textit{$(1+\varepsilon)$-isometry} for some $\varepsilon>0$ if 
\begin{equation}(1-\varepsilon)d(x,y)\leq \rho(\varphi(x),\varphi(y))\leq(1+\varepsilon)d(x,y)
\end{equation}
for all $x,y\in M$.

\smallskip
The Lebesgue measure on $\R$ is denoted by $\lambda$, and $\mathcal{H}^1$ denotes the Hausdorff $1$-measure.

\subsection{Lipschitz quotients}

In the definition of the co-Lipschitz constant of a map $f\colon M\rightarrow N$ at the start of this article, we opted to use closed balls, and, as a consequence, the infimum $\text{co-Lip}(f)$ is not always attained. However, if $M$ is a compact metric space, a standard argument shows that in fact $B(f(x),r)\subset f\big(B(x,\text{co-Lip}(f) \cdot r))$ for all $x\in M$ and $r>0$.

We will use the following elementary fact regarding Lipschitz quotients, whose proof we include for completeness. Given a metric space $M$, the completion of $M$ is denoted by $\overline{M}$.
\begin{lemma}
\label{Lemma:Quotients_between_subsets}
    Let $f:M\rightarrow N$ be a Lipschitz quotient mapping. Then there exists a Lipschitz quotient mapping $\overline{f}:\overline{M}\rightarrow \overline{N}$ so that $\overline{f}|_M=f$, $\coLip (f)=\coLip(\overline{f})$ and $\Lip(f)=\Lip(\overline{f})$.
\end{lemma}
\begin{proof}
    The existence of a Lipschitz map $\overline{f}\colon \overline{M}\rightarrow \overline{N}$ extending $f$ and with $\Lip(\overline{f})=\Lip(f)$ is a standard fact. Moreover, since $M$ is dense in $\overline{M}$, in order to show that $\text{co-Lip}(\overline{f})=\text{co-Lip}(f)$ it is enough to check that 
    \begin{equation}B(f(x),r)\subset \overline{f}(B(x,Cr))\end{equation}
    for all $C>\text{co-Lip}(f)$, $r>0$ and $x\in M$. Fix such $C,r$ and $x$, and consider $z\in B(f(x),r)$. There exists a sequence $(y_n)_n\subset N$ converging to $z$. Denote $r_1:=d(z,y_1)$ and $s_n:=d(y_n,y_{n+1})$ for all $n\in\N$. Since $f$ is a Lipschitz quotient, we know that for every $C_1>\text{co-Lip}(f)$ it holds that
    \begin{equation}B(f(x),r+r_1)\subset f(B(x,C_1(r+r_1)),\end{equation}
    and therefore there exists $x_1\in M$ such that $d(x,x_1)\leq C_1(r+r_1)$ and $f(x_1)=y_1$. Inductively, for every sequence $(C_n)_n$ with $C_n>\text{co-Lip}(f)$, we may define a sequence $(x_n)_n$ such that $f(x_n)=y_n$ and $d(x_n,x_{n+1})\leq C_ns_n$ for all $n\in \N$. The last condition implies that $(x_n)_n$ is a Cauchy sequence, and thus it converges to some $x_0\in \overline{M}$ which satisfies $\overline{f}(x_0)=z$ by continuity of $\overline{f}$. Finally, note that
    \begin{equation}d(x,x_0)\leq C_1r+C_1r_1+\sum_{n\in\N}C_ns_n,\end{equation}
    and thus with a suitable choice of $(C_n)_n$ and $(y_n)_n$ we can obtain $d(x,x_0)\leq Cr$, proving the result.
\end{proof}

\subsection{Curve-flat functions}\label{Section:Prel_CF}

The curve-flat pseudometric $d_{cf}$ admits an equivalent functional definition from which it derives its name. A Lipschitz map $f:M\rightarrow N$ is \textit{curve-flat} if for every Lipschitz curve fragment $\gamma:K\rightarrow M$, 
\begin{equation}
\mathcal{H}^1\big(f\circ \gamma(K)\big)=0.
\end{equation}
Curve-flat maps also admit several different equivalent definitions. In this paper, we will use the following result, whose proof can be found in \cite[Proposition 2.1]{Flo+24} (note that in \cite{Flo+24} Lipschitz curve fragments are called Lipschitz curves).
\begin{proposition}
    Let $f\colon M\rightarrow N$ be a Lipschitz map between metric spaces. The following assertions are equivalent:
    \begin{enumerate}[label=(\roman*)]
        \item $f$ is curve-flat,
        \item for every Lipschitz function $g\colon N\rightarrow \R$ and every curve-fragment $\gamma\colon K\rightarrow M$, $\lambda((g\circ f\circ \gamma)(K))=0$,
        \item for every Lipschitz curve-fragment $\gamma\colon K\rightarrow M$, 
        \begin{equation}\lim_{\substack{y\rightarrow x\\ y\in K}}\frac{d((f\circ \gamma)(x),(f\circ \gamma)(y))}{|x-y|}=0
        \end{equation}
        $\lambda$-almost everywhere in $K$.
    \end{enumerate}
\end{proposition}
It is shown in \cite[Proposition 5.22]{AGPP22} that, as a consequence of Lebesgue differentiation theorem, 
\begin{equation}\label{equation:cf_pseudometric_def}
d_{cf}(x,y)=\sup\{|f(x)-f(y)|\colon f:M\rightarrow \mathbb{R} \text{ is }1\text{-Lipschitz and curve-flat}\}
\end{equation}
for all $x,y\in M$. 

Curve-flat functions were formally introduced in \cite{AGPP22}, but an earlier, very general result of Bate \cite[Lemma 3.4]{Bat20} implies in particular that curve-flat functions can be uniformly approximated by functions with small Lipschitz constant in small neighborhoods. In the theory of Lipschitz-free spaces, this can be applied in order to use flatness techniques originally employed by Kalton \cite{Kal04} that make curve-flat functions relevant in the field (see \cite{Flo+24}).

\medskip

Let us describe more precisely the transfinite sequence of curve-flat quotients associated to any metric space that we mentioned in the introduction. The pseudometric $d_{cf}$ is a metric in the quotient space $M_{cf}:=M/d_{cf}$, and the resulting metric space $M_{cf}=(M_{cf},d_{cf})$ is called the \emph{curve-flat quotient of $M$}. 

Using transfinite induction, we define for every ordinal $\alpha$ the curve-flat pseudometric $d_{cf}^\alpha$ and the corresponding quotient $(M_{cf}^\alpha,d_{cf}^\alpha)$. More precisely, for an ordinal $\alpha+1$, we define $d_{cf}^{\alpha+1}:=(d_{cf}^\alpha)_{cf}$ and $M_{cf}^{\alpha+1}:=(M_{cf}^{\alpha})_{cf}$. For a limit ordinal $\alpha$, we define the pseudometric
\begin{equation}
d_{cf}^\alpha (x,y):=\inf_{\beta<\alpha} d_{cf}^\beta(x,y)
\end{equation}
and the corresponding quotient $M_{cf}^\alpha:=(M/d_{cf}^\alpha,d_{cf}^\alpha)$. The sequence $(M_{cf}^\alpha)_{\alpha}$ stabilizes at the \emph{curve-flat index} $\alpha_{cf}(M)$, which is the smallest ordinal $\alpha$ for which $M_{cf}^{\alpha}$ is purely $1$-unrectifiable; as clearly $X_{cf}$ is isometric to $X$ if and only if $X$ is purely $1$-unrectifiable. In fact, by Corollary 2.4 in \cite{Flo+24}, $X$ is purely $1$-unrectifiable as soon as $X_{cf}$ is bi-Lipschitz equivalent to $X$.
\smallskip

Any Lipschitz function between metric spaces induces a Lipschitz function with the same Lipschitz constant between the corresponding curve-flat quotients. More precisely, given a Lipschitz map $f\colon (M,d)\rightarrow (N,\rho)$ and an ordinal $\alpha$, there exists a map $f_\alpha\colon (M_{cf}^\alpha,d_{cf}^\alpha)\rightarrow(N_{cf}^\alpha,\rho_{cf}^\alpha)$ with $\Lip(f_\alpha)=\Lip(f)$ such that $q_N^\alpha(f(p))=f_{\alpha}(q_M^\alpha(p))$, where $q_X^\alpha\colon X\rightarrow X_{cf}^\alpha$ is the map that sends each point in a metric space $X$ to its equivalence class in $X_{cf}^\alpha$. We refer to \cite[Section 5]{AGPP22} for more details on these induced maps. 

\smallskip

\subsection{Lipschitz curves and the intrinsic distance}
A \textit{Lipschitz curve} in a metric space $M$ is a Lipschitz curve fragment $\gamma\colon K\rightarrow M$ where $K$ is an interval. Recall that a metric space $M$ is \textit{rectifiably connected} if every pair of points can be joined by a rectifiable curve, that is, a continuous curve of finite length. Since rectifiable curves can be parametrized by their arc-length, a metric space is rectifiably connected if and only if any pair of points can be joined by a Lipschitz curve. It is straightforward to show that a curve-flat function is constant along every Lipschitz curve. Therefore, curve-flat functions on rectifiably connected spaces are necessarily constant and the pseudometric $d_{cf}$ is trivially $0$ on such spaces. 

Given a rectifiably connected metric space $(M,d)$, we can define another distance $d_\Gamma$, called the \emph{intrinsic distance}, where $d_\Gamma(p,q)$ is the infimum of the length of all rectifiable curves joining $p$ and $q$. A metric space $(M,d)$ is \emph{length} if $d=d_\Gamma$, and it is \emph{geodesic} if the infimum defining $d_\Gamma$ is attained for every $p\neq q$. It is well known that a compact length space is necessarily geodesic. 
\subsection{Distortion functions} 

A function $\omega:[0,\infty)\rightarrow [0,\infty)$ is a \emph{distortion} if it is concave and satisfies $\omega(0)=0$. We say that a distortion $\omega$ is \textit{local} if it satisfies $\omega'(0)=\infty$. Given a metric $d$ and a distortion $\omega$, it is easy to check that $\omega\circ d$ is also a metric. Well-known examples of local distortions include the snowflake distortions, defined by $\omega(t)=t^\alpha$ for $0<\alpha<1$.
Lastly, we will extensively use that whenever $M$ is a metric space and $\omega$ is a local distortion, then the space $(M,\omega \circ d)$ is purely $1$-unrectifiable. For more background on (local) distortions, we refer to \cite[Chapter~2.6]{Wea18Book}.

\section{A separable metric space universal for $1$-Lipschitz quotients onto complete separable metric spaces}\label{Section:Univ_Lips_Quotients_Sep}

In this section we construct a separable metric space $SUT$ which is universal for $1$-Lipschitz quotients onto all separable complete metric spaces.

We modify the construction of an ``$\ell_1$ tree" from \cite[Section~2]{JLPS02}. In \cite{JLPS02}, $\ell_1$ trees were used to obtain complete metric spaces that have every complete separable \emph{geodesic} metric space as a $1$-Lipschitz quotient. However, as mentioned above, since rectifiability is preserved through continuous images, these metric trees cannot be mapped onto all separable metric spaces. Still, the main idea behind their construction can be adapted to be fully universal.

For a family of metric spaces $\{(M_i,d_i)\}_{i\in I}$ whose pairwise intersection is a single point $p$, we may define the \emph{$\ell_1$-sum} of $\{(M_i,d_i)\}_{i\in I}$ as the metric space $(\bigcup_{i\in I}M_i,d) $, where $d$ is the metric that agrees with $d_i$ in each $M_i$, and satisfies $d(x,y)=d_i(x,p)+d_j(p,y)$ for $x\in M_i$ and $y\in M_j$, $i\neq j$. We define the space $SUT$ as the completion of $\bigcup_{n\in\N}T_n$, where each $T_n$ is a countable metric space defined inductively. Define $T_0:=\{0\}$. Suppose $T_n=(x_k)_{k\in\N}$ has been defined for some $n\in\N$. Define $T_*$ as the $\ell_1$-sum of the two-point spaces $(\{0,q\},|\cdot|)_{q\in (0,\infty)\cap \Q}$, where $|\cdot|$ denotes the usual real line metric. For each $k\in\N$, we may consider the space $T_*^{x_k}$ to be an isometric copy of $T_*$ containing the point $x_k$, where $x_k\in T_*^{x_k}$ is identified with $0\in T_*$. Writing $T_{n+1}^0:=T_n$, we define $T_{n+1}^k$ as the $\ell_1$-sum of $T_{n+1}^{k-1}$ with $T_*^{x_k}$ for every $k\in \N$, and then $T_{n+1}:=\bigcup_{k\in \N}T_{n+1}^k$. Since $T_n$ is countable for every $n\in\N$, the union $\bigcup_{n\in\N}T_n$ is countable and its completion $SUT$ is separable. We call the point $0\in T_0\subset SUT$ the \emph{root} of $SUT$.

We are ready to prove Theorem \ref{Main_Theorem_1_SUT}.

\begin{theorem}
    The space $SUT$ is universal for $1$-Lipschitz quotients onto all separable complete metric spaces.
\end{theorem}
\begin{proof}
    Let $M$ be a separable metric space, and let $D\subset M$ be a countable dense subset. By Lemma \ref{Lemma:Quotients_between_subsets}, it is enough to construct a Lipschitz quotient from $T_\infty:=\bigcup_{n\in\N}T_n$ onto $D$. We will find a $1$-Lipschitz map $f\colon T_\infty\rightarrow D$ such that for every $n,k\in\N$, every $x\in T_n$ and every $r>0$ it holds that
    \begin{equation}
    \label{eq:f_quotient_at_dense_subset}
        B(f(x),r)\subset f(B(x,(1+2^{-k})r)),
    \end{equation}
    from which the conclusion follows. 
    We construct $f$ inductively, starting with the root $0\in SUT$. Choose any $a_0\in D$, and define $f(0)=a_0$. Suppose that $f$ has been defined for every point in $T_n\subset SUT$ for some $n\in\N$. Fix $x\in T_n$, and denote $a_x=f(x)$. For each $a\in D\setminus \{a_x\}$ and each $k\in \N$, we can find a rational number $q_{(a,k)}\in (0,+\infty)$ such that
    \begin{equation}
    \label{eq:choice_of_q(a,k)}
        (1+2^{-k})^{-1}q_{(a,k)}\leq d(a_x,a)\leq q_{(a,k)}.
    \end{equation}
    Moreover, since $D$ is countable, we can ensure that $q_{(a,k)}\neq q_{(b,l)}$ for $(a,k)\neq (b,l)\in D\setminus \{a_x\}\times\N$. Using the same definition of $T_*$ as above, denote by $T_*^x$ the isometric copy of $T_*$ contained in $T_{n+1}$ that was $\ell_1$-added to $T_n$ by the point $x$. For each rational number $q\in (0,+\infty)$, denote by $q^x$ the point in $T_*^x$ at distance $q$ from $x$. Then, we define $f(q^x):=a$ if $q=q_{(a,k)}$ for some $a\in D\setminus\{a_x\}$ and $k\in \N$, and $f(q^x):=a_x$ otherwise.

    After this process is done for every $x\in T_n$, $f(y)$ is defined for all $y\in T_{n+1}$. The second part of equation \eqref{eq:choice_of_q(a,k)} and the definition of the $\ell_1$-sum metric ensures that $f$ is $1$-Lipschitz in $T_{n+1}$. On the other hand, with the first part we obtain that equation \eqref{eq:f_quotient_at_dense_subset} holds for all $k\in \N$. This completes the inductive argument and the conclusion follows. 
\end{proof}

As we mentioned above, the space $SUT$ is purely $1$-unrectifiable. This can be shown, for instance, by noting that countably valued (and thus curve-flat) functions separate points in $SUT$ to their full distance. Indeed, given $x\neq y\in SUT$, the metric segment 
\begin{equation}[x,y]:=\{z\in SUT\colon d(x,y)=d(x,z)+d(z,y)\}\end{equation} is a countable set which is moreover a $1$-Lipschitz retract of $SUT$, with a Lipschitz retraction $r\colon SUT\rightarrow [x,y]$. Therefore, the function $f_{x,y}\colon [x,y]\rightarrow \R$ given by $f_{x,y}(z)=d(x,z)$ is a $1$-Lipschitz function taking on countably many values, that can be extended to a $1$-Lipschitz function $f_{x,y}\circ r$ defined on $SUT$ that also takes on countably many values. This resulting function must then be curve-flat, and still separates $x$ and $y$ to their full distance. Therefore, the quotient space $SUT_{cf}$ (see Section \ref{Section:Prel_CF}) is isometric to $SUT$, which implies that $SUT$ is purely $1$-unrectifiable.

\section{Universality for Lipschitz quotients: The compact case}\label{Section:Lipschitz_quotients}

We now focus on the compact case. As hinted at in the introduction, compactness can be used in the domain of a Lipschitz quotient $f\colon M\rightarrow N$ in order to construct curve-fragments that join points in $M$, with gaps of comparable size to those in curve-fragments in $N$. The following Proposition is a quantitative expression of this idea, that works for high-order curve-flat pseudometrics.

Recall that, in compact spaces, the co-Lipschitz constant is in fact attained. This fact is not essential in our arguments but it simplifies computations.
\begin{proposition}
\label{Prop:technical_ineq_quotients}
    Let $\alpha<\omega_1$. Let $(M,d)$ be a compact metric space, let $(N,\rho)$ be a metric space, and let $f\colon (M,d)\rightarrow (N,\rho)$ be a Lipschitz quotient map with co-Lipschitz constant $C$. Then, for every $x\in M$ and every $y\in N$ there exists $p\in M$ such that $f(p)=y$ and 
    \begin{equation}\label{Eq:Lip_quot_CF}
 d^\alpha_{cf}(x,p)\leq C\rho^\alpha_{cf}(f(x),y)=C\rho^\alpha_{cf}(f(x),f(p)).
    \end{equation}
    In particular, the induced map $f_\alpha\colon (M/d_{cf}^\alpha,d_{cf}^\alpha)\rightarrow (N/\rho_{cf}^\alpha,\rho_{cf}^\alpha)$ (see Section \ref{Section:Prel_CF}) is a Lipschitz quotient map with co-Lipschitz constant $C$. 
\end{proposition}

\begin{proof}
   We prove this Proposition using transfinite induction. First, assume $\alpha=1$. 
   
  Fix $x\in M$ and $y\in N$. For every $\varepsilon>0$ we find a compact $\Ke\subset \R$ and a $1$-Lipschitz function $\psi:\Ke\rightarrow N$ with $\psi(\min(\Ke))=f(x)$ and  $\psi(\max(\Ke))=y$ such that 
    \begin{equation} \lambda([\min \Ke,\max \Ke]\setminus \Ke)\leq\rho_{cf}(f(x),y)+\varepsilon.\end{equation} 
Let $D:=\{a_n\colon n\in \N\}$ be a dense subset of $\Ke$ with $a_1=\min \Ke$.

    Fix $n\in \N$. Let $z_1^n,\dots,z_n^n$ in $\R$ and let $\pi_n:\{1,\ldots,n\}\rightarrow\{1,\ldots,n\}$ be a permutation such that $z_1^n< z_{2}^n<\ldots<z_n^n$ and $z_{\pi_n(i)}^n=a_i$ for all $i\in \{1,\ldots n\}$. Note that $z_1^n=a_1$ since $a_1=\min \Ke$. We will inductively define a sequence $(x_i^n)_{i=1}^n$ in $M$ such that $f(x_i^n)=\psi(z_i^n)$ and such that $d(x_{i+1}^n,x_i^n)\leq C(z_{i+1}^n-z_i^n)$ for all $i\in\{1,\dots,n\}$. 
    
    Define $x_1^n=x$. Having defined $x_1^n,\ldots,x_i^n\in M$ for some $i<n$, we next define $x_{i+1}^n \in M$. Since $f$ is co-Lipschitz with constant $C$, 
    \begin{equation}
    \label{equation:f_CoLip_at_x_in}
        B(f(x_i^n),z_{i+1}^n-z_{i}^n)\subset f(B(x_i^n,C(z_{i+1}^n-z_{i}^n)).
    \end{equation}
    Additionally, using that $\psi$ is $1$-Lipschitz and that $\psi(z_i^n)=f(x_i^n)$, we obtain that 
    \begin{equation}
    \label{equation:psi_1_Lip_at_z_in}
    \psi(z^n_{i+1})\in B(f(x_i^n),z_{i+1}^n-z_i^n).
    \end{equation}
    Combining equations \eqref{equation:f_CoLip_at_x_in} and \eqref{equation:psi_1_Lip_at_z_in}, we can find $x_{i+1}^n\in B(x_i^n,C(z_{i+1}^n-z_{i}^n))$ such that $f(x_{i+1}^n)=\psi(z_{i+1}^n)$, thus finishing the induction. Note that for all different $i,j\in\{1,\dots n\}$ with $i<j$, we have

\begin{equation}
\label{equation:Choice_of_x_in_is_C_Lip}
    d(x_i^n,x_j^n)\leq \sum_{k=i}^{j-1} d(x_k,x_{k+1})\leq\sum_{k=i}^{j-1} C(z_{k+1}^n-z_{k}^n)=C|z_i^n-z_j^n|.
\end{equation}

Now, we aim to define a sequence $(x_n)_n$ such that $f(x_n)=\psi(a_n)$ for all $n\in\N$ and such that $d(x_n,x_m)\leq C|a_n-a_m|$ for all $n,m\in \N$. We do so inductively. First, define $x_1=x$. Note that $x_1=x^n_1$ for all $n\in\N$. Next, consider the sequence $(x^n_{\pi_n(2)})_n$ in $M$. Notice that $f(x^n_{\pi_n(2)})=\psi(z^n_{\pi_n(2)})=\psi(a_2)$ for all $n\in\N$. Using compactness of $M$, find a converging subsequence $(x_{\pi_{n_k}(2)}^{n_k})_k$ with a limit $x_2\in M$. By continuity of $f$ we still have that $f(x_2)=\psi(a_2)$, and, using equation \eqref{equation:Choice_of_x_in_is_C_Lip}, we get that 
\begin{equation}d(x^{n_k}_{\pi_{n_k}(1)},x^{n_k}_{\pi_{n_k}(2)})\leq C|z^{n_k}_{\pi_{n_k}(1)}-z^{n_k}_{\pi_{n_k}(2)}|=C|a_1-a_2|,\end{equation}
for all $k\in\N$, and so, passing to the limit, we get $d(x_1,x_2)\leq C|a_1-a_2|$. For the next step, we consider the sequence $(x^{n_k}_{\pi_{n_k}(3)})_k$ and obtain a further subsequence that converges to some $x_3\in M$ with $f(x_3)=\psi(a_3)$ and that, thanks to equation \eqref{equation:Choice_of_x_in_is_C_Lip} satisfies
\begin{equation}d(x_i,x_3)\leq C|a_i-a_3|,\end{equation}
for $i=1,2$. Continuing in this way we obtain the desired sequence $(x_n)_n$.

Now we can define a C-Lipschitz map $\phi:D\rightarrow M$ by $\phi(a_n)=x_n$ and extend it to the whole space $\Ke$. Denote $p_\varepsilon=\phi(\max(\Ke))$.  Note that 
\begin{equation}
\label{equation:f(p_eps)=y}
f(p_\varepsilon)=\psi(\max(\Ke))=y
\end{equation}
and therefore
\begin{equation}
\label{equation:d_ur(x,p_eps)_estimate}
     d_{cf}(x,p_\varepsilon)\leq C\lambda([\min \Ke,\max \Ke]\setminus K)\leq C \rho_{cf}(f(x),y)+C\varepsilon.
\end{equation}

   Take any sequence $(\varepsilon_n)_n$ converging to $0$, consider the corresponding sequence $(p_{\varepsilon_n})_n$ in $M$, and find a converging subsequence with a limit $p\in M$. Combining equations \eqref{equation:f(p_eps)=y} and \eqref{equation:d_ur(x,p_eps)_estimate} we get that $f(p)=y$ and that equality $(\ref{Eq:Lip_quot_CF})$ holds. This proves the case $\alpha=1$.

   For a successor ordinal $\alpha+1$, it is enough to apply the base case to the map $ f_{\alpha}$, which is a Lipschitz quotient with co-Lipschitz constant $C$ by inductive hypothesis. 

    If $\alpha$ is a limit ordinal, recall that for any $x,y\in M$, we have
    \begin{equation}d_{cf}^\alpha (x,y)=\inf_{\beta<\alpha} d_{cf}^\beta(x,y) \qquad \text{and}\qquad \rho_{cf}^\alpha(f(x),f(y))=\inf_{\beta<\alpha}\rho_{cf}^\beta(f(x),f(y)).\end{equation}
 Since equality (\ref{Eq:Lip_quot_CF}) holds for all $\beta<\alpha$, then it also holds  for the corresponding infimum aswell.
\end{proof}

As a corollary, we get Theorem \ref{Main_theorem_p1u}.

\begin{corollary}
    \label{Corollary:Quotient_of_p1u_is_p1u}
    Let $(M,d)$ be a compact purely $1$-unrectifiable metric space and let $(N,\rho)$ be a Lipschitz quotient of $(M,d)$. Then $(N,\rho)$ is also purely $1$-unrectifiable.
\end{corollary}
\begin{proof}
    Let $f\colon (M,d)\rightarrow (N,\rho)$ be a Lipschitz quotient with co-Lipschitz constant $C$. Since $(M,d)$ is purely $1$-unrectifiable, it is isometric to its curve-flat quotient and we have that $d_{cf}=d$. We will show that $\rho_{cf}$ is a metric on $N$ and that $(N,\rho)$ and $(N,\rho_{cf})$ are bi-Lipschitz equivalent. 
    
    Fix $y\neq z\in N$. Since $f$ is surjective, there exists $x\in M$ such that $f(x)=z$, and applying Proposition \ref{Prop:technical_ineq_quotients} with $\alpha=1$, we find $p\in M$ such that $f(p)=y$ and 
    \begin{equation}d(x,p)\leq C\rho_{cf}(z,y),\end{equation}
    where we used that $d_{cf}=d$ in $M$. Since $f(x)\neq f(p)$, we obtain that $x\neq p$, and thus we get that $\rho_{cf}$ is indeed a metric on $N$. Finally, notice that applying the previous inequality and the fact that $f$ is Lipschitz we get
    \begin{equation}\rho_{cf}(z,y)\leq \rho(z,y)\leq \text{Lip}(f)d(x,p)\leq \text{Lip}(f)C \rho_{cf}(z,y),\end{equation}
    which proves that $(N,\rho)$ and $(N,\rho_{cf})$ are bi-Lipschitz equivalent. However, this implies that $N$ is purely $1$-unrectifiable (see e.g: Corollary 2.4 in \cite{Flo+24}).
\end{proof}

We are now ready to prove Theorem \ref{Main_Theorem_index_through_quotients}.

\begin{theorem}
\label{Theorem:curve_flat_index_quotients}
    Let $(M,d)$ be a compact metric space and let $(N,\rho)$ be a Lipschitz quotient of $(M,d)$. Then $\alpha_{cf}(M)\geq\alpha_{cf}(N)$. 
\end{theorem}
\begin{proof}
    Let $f\colon (M,d)\rightarrow (N,\rho)$ be a Lipschitz quotient map and denote $\alpha=\alpha_{cf}(M)$. 
    By Proposition \ref{Prop:technical_ineq_quotients}, the induced map $f_\alpha\colon (M/d_{cf}^\alpha,d_{cf}^\alpha)\rightarrow (N/\rho_{cf}^\alpha,\rho_{cf}^\alpha)$ is a Lipschitz quotient. Since $(M/d_{cf}^\alpha,d_{cf}^\alpha)$ is purely $1$-unrectifiable, the result follows from Corollary \ref{Corollary:Quotient_of_p1u_is_p1u}.
\end{proof}

\medskip

Now, in order to prove Theorem \ref{Main_Theorem_2_nocompactSUT} we only need to show the existence of compact metric spaces with arbitrarily high countable curve-flat index. This follows from Theorem \ref{Main_Theorem_4_compacthighordercf}, whose technical proof we postpone to the following sections. Now, we assume its validity and we show a formally stronger version of Theorem \ref{Main_Theorem_2_nocompactSUT}.

\begin{theorem}
\label{Theorem:main_theorem_3_general}
    Given a compact purely $1$-unrectifiable metric space $N$, consider $\mathcal{F}_N$ the family of all compact metric spaces $K$ such that $K_{cf}^{\omega_1}$ is bi-Lipschitz equivalent to $N$. There does not exist a compact metric space $M$ universal for Lipschitz quotients onto $\mathcal{F}$.

    In particular, there does not exist a compact metric space universal for Lipschitz quotients onto all compact metric spaces.
\end{theorem}
\begin{proof}
    Suppose by contradiction that $M$ is a compact space such that for every compact space $K\in \mathcal{F}_N$ there exists a Lipschitz quotient from $M$ onto $K$. Since $M$ is compact, its curve-flat index $\alpha_{cf}(M)$ is a countable ordinal. By Theorem \ref{Main_Theorem_4_compacthighordercf}, we can find a compact metric space $Y\in \mathcal{F}_N$ such that $\alpha_{cf}(Y)>\alpha_{cf}(M)$. However, $Y$ is a Lipschitz quotient of $M$ so $\alpha_{cf}(Y)\leq\alpha_{cf}(M)$ by Theorem \ref{Theorem:curve_flat_index_quotients}; a contradiction. 
\end{proof}

Inspecting the previous proof, one may be tempted to claim that in order to prove Theorem \ref{Main_Theorem_2_nocompactSUT}, we only need to show that for every compact metric space $M$ there exists a compact $Y$ such that $Y/d_{cf}$ is bi-Lipschitz equivalent to $M$, and thus Theorem \ref{Main_Theorem_4_compacthighordercf} is not needed in its full generality for this purpose. However, note that for such $M$ and $Y$, it is not necessarily true that $\alpha_{cf}(Y)>\alpha_{cf}(M)$. Indeed, for such a pair it only follows that $\alpha_{cf}(Y)=1+\alpha_{cf}(M)$, which coincides with $\alpha_{cf}(M)$ if the latter is infinite. In fact, we need Theorem \ref{Main_Theorem_4_compacthighordercf} to work for any arbitrarily countable ordinal $\alpha$, since \emph{additively indecomposable ordinals} (that is, ordinals $\gamma$ such that $\beta+\gamma=\gamma$ for all $\beta<\gamma$) are precisely those of the form $\omega^\beta$ for any ordinal $\beta$ (see e.g: \cite[Exercise 2.13]{JechBook}).

\section{Finite order curve-flat quotients}\label{Section:Finite_order}

This section is a preamble to Section \ref{Section:High_order}, and its goal is to prove Theorem \ref{Main_Theorem_3_finiteorder_CF}. Here, we introduce the \emph{gapped graph} construction that lies at the core of the proof of both Theorems \ref{Main_Theorem_4_compacthighordercf} and \ref{Main_Theorem_3_finiteorder_CF}.

We start by formalizing the notion of the \emph{geodesic graph generated by a metric space}.

\begin{definition}[Geodesic graph]
    Let $(M,d)$ be a metric space. Let $X$ be a Banach space such that $M$ is a linearly independent subset of $X$ (we may take, for example, $X$ to be the Lipschitz-free space of $M$). The \textit{geodesic graph generated by $M$} is the metric space $(G(M),d_G)$, with 
    \begin{equation}G(M):=\bigcup_{x,y\in M}[x,y]\subset X,\end{equation}
    where $[x,y]:=\{\lambda y+(1-\lambda)x\colon \lambda\in [0,1]\}\subset X$; and $d_G$ is the intrinsic metric associated to $G(M)$ as a subset of the Banach space $X$.
\end{definition}    

Intuitively, the space $(G(M),d_G)$ is the shortest path metric on the complete weighted graph (including the edges in the set $G(M)$) whose set of vertices is $M$ and every pair of points $x\neq y\in M$ is connected by an edge of weight $d(x,y)$. Note that if $M$ is complete, its geodesic graph is also complete.

The metric spaces we are interested in are closed subsets of the geodesic graph generated by a distortion of $M$. We will mainly use \emph{local distortions}, but the following definition makes sense for a bigger class of distortion functions.

\begin{definition}[Gapped graph]
    Let $(M,d)$ be a metric space, let $\omega$ be a distortion function such that $\omega(t)\geq t$ for all $t\in [0,+\infty)$, and let $D\subset M$. The \emph{gapped graph associated to $(M,d)$, $\omega$ and $D$} is the subspace of $(G(M),(\omega\circ d)_G)$ given by 
    \begin{equation}G(M,\omega,D):=M\cup \bigcup_{x,y\in D}[x,u]\cup[v,y],\end{equation}
    where $u,v$ are the unique points in the edge $[x,y]$ such that $(\omega\circ d)_G(u,v)=d(x,y)$ and $(\omega\circ d)_G(x,u)=(\omega\circ d)_G(v,y)$.
\end{definition}

Intuitively, the gapped graph associated to $(M,d)$, $\omega$ and $D$ is the subset of the geodesic graph generated by $(M,\omega\circ d)$ consisting of every vertex, and those edges $[x,y]$ with $x,y\in D$, removing a gap in the middle of $[x,y]$ of length exactly $d(x,y)$. The exact position of the gap of length $d(x,y)$ in the segment $[x,y]$ is not relevant for our purposes.

It is direct to check that $G(M,\omega,D)$ is a closed subset of $(G(M),(\omega\circ d)_G)$, and hence it is a complete metric space whenever $M$ is complete. 

A gapped graph $G(M,\omega,D)$ can be decomposed as a disjoint union of rectifiably connected (indeed geodesic) components, each containing a point of $M$. Given $x\in M$, we call $\text{Arm}(x)$ \emph{the arm of $x$ in $G(M,\omega,D)$} as the unique rectifiably connected component of $G(M,\omega,D)$ containing the point $x$. Note that, if $x\neq y\in D$, it holds that $(\omega\circ d)_G(\text{Arm}(x),\text{Arm}(y))=d(x,y)$. Figure \ref{fig:Gapped_graph_of_3_points} shows an example of the gapped graph of a finite metric space.

\begin{figure}
    \centering
    \includegraphics[width=0.7\linewidth]{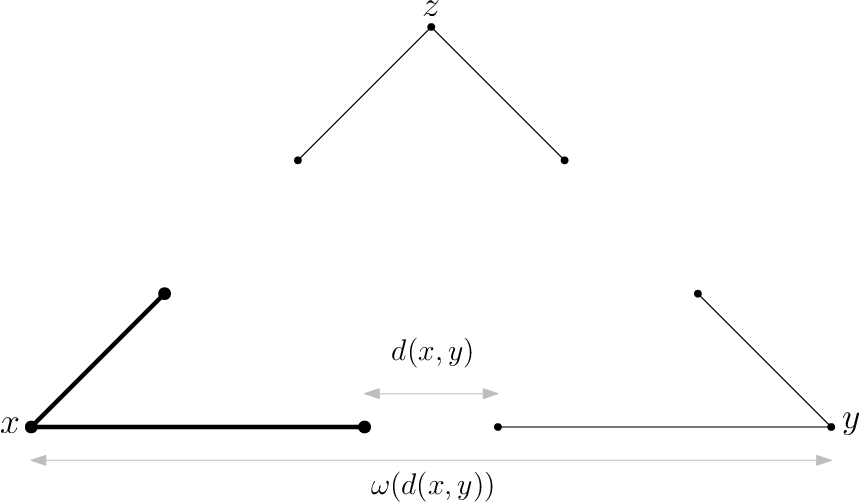}
    \caption{The gapped graph $G(M,\omega,D)$ of a three point metric space $(M,d)$, with $D=M$. The set in bold represents the arm of $x$.}
    \label{fig:Gapped_graph_of_3_points}
\end{figure}

The following proposition characterizes curve-flat functions in gapped graphs.

\begin{proposition}
    
\label{Proposition:Main_Theorem_A_general}
    Let $(M,d)$ be a complete metric space, let $\omega$ be a distortion function with $\omega(t)\geq t$ for all $t\in[0,+\infty)$ and let $D\subset M$. Let $f\colon (G(M,\omega,D),(\omega\circ d)_G)\rightarrow \R$ be a Lipschitz function. The following assertions are equivalent:

    \begin{enumerate}[label=(\alph*)]
        \item The function $f$ is curve-flat.\label{MainPropositionAGen_a}
        \item The restriction $f_{|(M,\omega\circ d)}$ is curve-flat and $f$ is constant on every arm.\label{MainPropositionAGen_b}
    \end{enumerate}

\end{proposition}
\begin{proof}
    The implication \ref{MainPropositionAGen_a} $\Rightarrow$ \ref{MainPropositionAGen_b} is straightforward. Indeed, the restriction of a curve-flat function is curve-flat, and thus it follows directly that $f_{|(M,\omega\circ d)}$ is curve-flat. Similarly, for every $x\in M$, the function $f_{|\text{Arm}(x)}$ is curve-flat, and since $\text{Arm}(x)$ is geodesic, this function must be constant.

    To show \ref{MainPropositionAGen_b} $\Rightarrow$ \ref{MainPropositionAGen_a}, we prove that $\lambda((f\circ\varphi)(K))=0$ whenever $K\subset \R$ is a compact subset and $\varphi\colon K\rightarrow G(M,\omega,D)$ is a Lipschitz map. Fix such $K$ and $\varphi$, and denote $K_M:=\varphi^{-1}(M)$.
    
    Since $K\setminus K_M$ is separable and $\varphi^{-1}(\text{Arm}(x)\setminus M)$ is open for all $x\in M$, there exists a countable set $(x_n)_n\subset M$ such that 
    \begin{equation}K=K_M\cup\bigcup_{n\in \N}K_{n}\end{equation}
    where $K_{n}:=\varphi^{-1}(\text{Arm}(x_n)\setminus M).$
    Since $f_{|(M,\omega\circ d)}$ is curve-flat, it follows that $\lambda((f\circ \varphi)(K_M))=0$, and since $f$ is constant on every arm, we also obtain that $\lambda((f\circ\varphi)(K_{n}))=0$ for all $n\in\N$. In conclusion,
    \begin{equation}\lambda((f\circ\varphi)(K))= \lambda((f\circ \varphi)(K_M))+\sum_{n\in\N}\lambda((f\circ\varphi)(K_{n}))=0 .\qedhere\end{equation}
\end{proof}

Now we can prove Theorem \ref{Main_Theorem_3_finiteorder_CF}, as a particular case of the following result.

\begin{theorem}
    \label{Theorem:MainTheoremA_general}
    Let $(M,d)$ be a complete metric space, let $\omega$ be a local distortion, and let $D\subset M$ be a dense subset. Then the curve-flat quotient of the gapped graph $(G(M,\omega,D),(\omega\circ d)_G)$ is isometric to $(M,d)$.
\end{theorem}
\begin{proof}
    We prove that for every $x,y\in M$,
    \begin{equation}
    \label{eq:main_eq_of_mt1}
        ((\omega\circ d)_G)_{cf}(p,q)=d(x,y)
    \end{equation}
    for any $p,q\in G(M,\omega,D)$ such that $p\in \text{Arm}(x)$ and $q\in \text{Arm}(y)$. Since $D$ is dense in $(M,d)$ and $\omega$ is a local distortion, it is also dense in $(M,\omega\circ d)$. Therefore, it is enough to prove that equation \eqref{eq:main_eq_of_mt1} holds for $x,y\in D$. 

    Fix $x,y\in D$, $p\in \text{Arm}(x)$ and $q\in \text{Arm}(y)$. Let $f_x\colon G(M,\omega,d)\rightarrow \R$ be the $1$-Lipschitz function given by $f_x(s)=d(x,z)$, where $z\in M$ is the unique point such that $s\in \text{Arm}(z)$. The function $f$ is clearly constant on every arm of $G(M,\omega,d)$, and $f_{|(M,\omega\circ d)}$ is curve-flat since $\omega$ is a local distortion and thus $(M,\omega\circ d)$ is purely $1$-unrectifiable. Therefore, by Proposition \ref{Proposition:Main_Theorem_A_general} $f_x$ is curve-flat and hence 
    \begin{equation}
        ((\omega\circ d)_G)_{cf}(p,q)\geq |f_x(p)-f_x(q)|=d(x,y),
    \end{equation}
    using equation \eqref{equation:cf_pseudometric_def}. 
    
    On the other hand, again by Proposition \ref{Proposition:Main_Theorem_A_general}, every curve-flat function $f\in \Lip(G(M,\omega,d))$ is constant on every arm. It follows that
    \begin{equation}((\omega\circ d)_G)_{cf}(p,q)\leq (\omega\circ d)_G(\text{Arm}(x),\text{Arm}(y))=d(x,y),
    \end{equation}
    and equation \eqref{eq:main_eq_of_mt1} holds. Now, we get that if $p,q\in G(M,\omega,D)$ their curve-flat distance is $0$ if and only if they belong to the same arm. Therefore, the curve-flat quotient of $(G(M,\omega,D),(\omega\circ d)_G)$ is bijective to $M$, and by \eqref{eq:main_eq_of_mt1} the curve-flat distance in this quotient is isometric to the distance $d$ in $M$. 
    
\end{proof}

As we mentioned before, Theorem \ref{Main_Theorem_3_finiteorder_CF} follows from Theorem \ref{Theorem:MainTheoremA_general} by taking a dense set $D$ with minimal cardinality and a local distortion satisfying $\omega(t)>t$ for all $t\in(0,+\infty)$, in the case of non-trivial metric spaces. For the one point metric space, the result is obvious since every rectifiably connected metric space has a one point curve-flat quotient.

\section{The compact case: High order curve-flat quotients}\label{Section:High_order}

In this section we prove Theorem \ref{Main_Theorem_4_compacthighordercf}. Given a triplet $(M,\alpha,\varepsilon)$, where $M$ is a compact metric space, $\alpha$ is a countable ordinal, and $\varepsilon>0$, our goal is to construct another compact metric space whose curve-flat quotient of order $\alpha$ is $(1+\varepsilon)$-isometric to $M$.

We construct such spaces by transfinite induction. Note that by directly applying the results in the previous section we can easily get the result for finite ordinals, but the resulting metric spaces are not compact unless $M$ is finite.  In addition, the (first) limit ordinal step cannot be easily deduced. 

In order to obtain a construction that works for all countable ordinals, we go back to the previous section with a slightly different perspective. Previously, we constructed the gapped graph $G(M,\omega,D)$ as a subset of a complete geodesic graph generated by a local distortion $(M,\omega\circ d)$. However, we can also obtain it by sequentially gluing in segments $S_{x,y}$ of length $\omega(d(x,y))$ with a gap of length $d(x,y)$ to countably many pairs of points $(x,y)\in (M,\omega\circ d)\times (M,\omega\circ d)$. Such \emph{gapped segments} $S_{x,y}$ clearly satisfy that their first-order curve-flat quotient is isometric to the pair $(\{x,y\},d)$, which is ultimately the condition that forces the first-order curve-flat quotient of $G(M,\omega,D)$ to be isometric to $(M,d)$. Hence, a natural way of adapting this construction to a given countable ordinal $\alpha$ is to, instead, glue in compact metric spaces $S_{x,y}^\beta$ for $\beta<\alpha$ (inductively constructed) such that $S_{x,y}^\beta/\rho_{cf}^\beta$ is isometric to $(\{x,y\},d)$. Since we are gluing in spaces at infinitely many pairs, we can obtain any countable ordinal index by using a suitable sequence of ordinals $(\beta_n)_n<\alpha$.

\subsection{Pseudometrics, attachments, and bendings}

The resulting construction is fairly technical. In order to simplify it, we will work with pseudometrics, since they allow us to avoid taking quotients until the end of the process, letting us work with fewer spaces overall. For this reason, we start by extending some of the definitions in the preliminaries to the pseudometric setting. 

Let $(M,\rho)$ be a pseudometric space. Then, the quotient $M/\rho$ equipped with $\rho$ is a metric space.
Moreover, if $f\colon (M,\rho)\rightarrow (N,d)$ is a Lipschitz function into a metric space $(N,d)$, it holds that $f(p)=f(q)$ for all $p,q\in M$ such that $\rho(p,q)=0$. Hence, the induced function $\overline{f}\colon (M/\rho,\rho)\rightarrow (N,d)$ given by $\overline{f}([p])=f(p)$ for all $p\in M$ is well defined and has the same Lipschitz constant as $f$.  

With this, we can suitably generalize the definition of curve-flat functions to pseudometric spaces in the following way: We say that a Lipschitz function $f\colon (M,\rho)\rightarrow \R$ is curve-flat if the induced function $\overline{f}\colon (M/\rho,\rho)\rightarrow \R$ is curve-flat. This definition is equivalent to the more direct generalization of curve-flatness to pseudometric spaces: a Lipschitz function $f\colon (M,\rho)\rightarrow \R$ is curve-flat if and only if for every Lipschitz curve fragment $\gamma\colon K\rightarrow (M,\rho)$, it holds that $\lambda((f\circ\gamma)(K))=0$. 
Analogously to equation \eqref{equation:cf_pseudometric_def}, we can define a curve-flat pseudometric $\rho_{cf}^\alpha\colon M\times M\rightarrow[0,+\infty)$ for every ordinal $\alpha$, without the need to define the intermediate quotient spaces. Moreover, $\rho_{cf}^\alpha$ defined in this way coincides with the original curve-flat pseudometric of order $\alpha$ whenever $\rho$ is a metric.

At the end of our construction, we will finally take the quotient $M/\rho_{cf}^\alpha$ and endow it with $\rho_{cf}^\alpha$, which is a metric on $M/\rho_{cf}^\alpha$, and we will show that this quotient is (almost) isometric to a certain compact space. For this reason, it is also useful for us to codify when a metric space $(M,d)$ is (almost) isometric to a quotient $(N/\rho,\rho)$ only considering the pseudometric space $(N,\rho)$.

We say that a subspace $S$ of a pseudometric space $(N,\rho)$ is \emph{total} if for every $x\in N$ there exists $y\in S$ such that $\rho(x,y)=0$. With this, it is easy to observe that an injective Lipschitz map $\varphi\colon (M,d)\rightarrow (N,\rho)$ between a metric space $(M,d)$ and a pseudometric space $(N,\rho)$ induces a $(1+\varepsilon)$-isometry $\overline{\varphi}\colon (M,d)\rightarrow (N/\rho,\rho)$ for $\varepsilon>0$ if and only if $\varphi(M)$ is total in $(N,\rho)$ and 
\begin{equation}(1-\varepsilon)d(x,y)\leq \rho(\varphi(x),\varphi(y))\leq(1+\varepsilon)d(x,y)\end{equation}
for all $x,y\in M$.

We will often consider different instances of a single function by endowing the domain with different pseudometrics. Therefore, we will write that $f\colon (M,\rho)\rightarrow \R$ has property P (such as Lipschitz, curve-flat, etc.) to mean that property P holds if $M$ is endowed with the pseudometric $\rho$. However, property P might fail when endowing $M$ with a different pseudometric.

We will need a basic result regarding curve-flat functions, which can be shown with an argument essentially contained in \cite[Proposition 5.22]{AGPP22}. We include its proof for pseudometric spaces.

\begin{lemma}
\label{Lemma:CF_pseudometric_is_CF}
    Let $(M,\rho)$ be a pseudometric space and let $\alpha$ be an ordinal. If a function $f\colon (M,\rho^\alpha_{cf})\rightarrow \R$ is Lipschitz, then $f\colon (M,\rho_{cf}^\beta)\rightarrow \R$ is curve-flat for all $\beta<\alpha$.
    
    In particular, the map $\rho_{cf}^\alpha(p,\cdot)\colon(M,\rho_{cf}^\beta)\rightarrow \R$  is curve-flat for every $\beta<\alpha$ and for every $p\in M$.
\end{lemma}

\begin{proof}
    We will prove it for $\alpha=1$, as the general case follows easily by transfinite induction. Let $f\colon (M,\rho_{cf})\rightarrow \R$ be a $1$-Lipschitz function. We will show that, for every compact subset $K\subset \R$ and every $1$-Lipschitz map $\gamma\colon K\rightarrow M$ 
    \begin{equation}\lim_{\substack{t\rightarrow t_0\\ t\in K}}\frac{|f(\gamma(t))-f(\gamma(t_0))|}{|t-t_0|}=0\end{equation}
    for almost every $t_0\in K$. Fix $t_0\in K$ and $t\neq t_0$. Since $f$ is $1$-Lipschitz with respect to the pseudometric $\rho_{cf}$, applying the definition of $\rho_{cf}$ for the curve-fragment $\gamma_{[t,t_0]}$, we get that
    \begin{align}
        |f(\gamma(t))-f(\gamma(t_0))|&\leq \rho_{cf}(\gamma(t),\gamma(t_0))\leq \lambda([t_0,t]\setminus K),
    \end{align}
    where $[t_0,t]$ denotes also the interval $[t,t_0]$ if $t<t_0$. Now, if $t_0$ is a density point in $K$ for the Lebesgue measure, then
    \begin{equation}\lim_{t\rightarrow t_0}\frac{|f(\gamma(t))-f(\gamma(t_0))|}{|t-t_0|}\leq\lim_{t\rightarrow t_0} \frac{\lambda([t_0,t]\setminus K)}{|t-t_0|}=0\end{equation}
    By Lebesgue density theorem, the set of density points in $K$ has full measure, and thus the result follows. 
\end{proof}

Moving on, we need three main technical definitions. The first is a formalization of the gluing process. This is a standard technique which can be implemented in many different ways. In this case, we choose a formulation close to the one in \cite{HQ2023}, but adapted to the pseudometric approach.

\begin{definition}[Attachment]
Let $(M,\rho_M)$ be a complete pseudometric space, called the \emph{frame}. Let $\mathcal{S}=\{(S_\gamma,\rho_\gamma)\}_{\gamma\in \Gamma}$ be a collection of complete pseudometric spaces, called the \emph{threads} such that:
\begin{itemize}
    \item $S_{\gamma}\cap S_\eta\subset M$ for all $\gamma\neq\eta$.
    \item For every $\gamma\in \Gamma$, the pseudometrics $\rho_X$ and $\rho_\gamma$ coincide in the set $Z_\gamma:=M\cap S_\gamma$, which is called the \emph{anchor of $S_\gamma$}.
\end{itemize}
The \emph{attachment of the frame $M$ with the threads $\mathcal{S}$} is the pseudometric space $(M(\mathcal{S}),\rho_\mathcal{S})$, where 
\begin{equation}M(\mathcal{S})=M\cup\bigg(\bigcup_{\gamma\in\Gamma} S_\gamma\bigg), \end{equation}
and 
\begin{align}
   \rho_\mathcal{S}\colon &  M(\mathcal{S})\times M(\mathcal{S})\longrightarrow \mathbb{R}^+
\end{align} 
is the largest pseudometric on $M(\mathcal{S})$ that agrees with $\rho_M$ on $M$ and agrees with $\rho_\gamma$ on each $S_\gamma$. Concretely:
\begin{equation} 
\footnotesize
\rho_\mathcal{S}(p,q)=
\begin{cases}
\rho_X(p,q),&\text{ if } p,q\in M,\\
\rho_\gamma(p,q), &\text{ if }p,q\in S_\gamma\text{ for some }\gamma\in \Gamma,\\
\min_{x\in Z_\gamma}\{\rho_\gamma(p,x)+\rho_X(x,q)\}, &\text{ if }p\in S_\gamma\text{ for some }\gamma\in \Gamma,\text{ and }q\in X,\\
H_{\gamma,\eta}(p,q), &\text{ if } p\in S_{\gamma},~ q\in S_{\eta}\text{ for }\gamma\neq \eta\in \Gamma,
\end{cases}
\end{equation}
where 
\begin{equation} H_{\gamma,\eta}(p,q)=\min\{\rho_\gamma(p,x)+\rho_X(x,y)+\rho_\eta(y,q)\colon x\in Z_\gamma,~ y\in Z_\eta\}.\end{equation}
\end{definition}

The second thing we formally define is the gapped segment associated to a pair of points and a distortion function. 

\begin{definition}[Gapped segments]
    Let $(\{x,y\},d)$ be a two point metric space and let $\omega$ be a distortion function with $\omega(d(x,y))\geq d(x,y)$. The metric space $(S_{\{x,y\},\omega},d_{\{x,y\},\omega})$ is the subset of the real line interval $[0,\omega(d(x,y))]$ obtained by removing a segment of length $d(x,y)$ in the middle. That is, 
    \begin{equation}S_{\{x,y\},\omega}:=\left[0,\frac{\omega(d(x,y))-d(x,y)}{2}\right]\cup\left[\frac{\omega(d(x,y))+d(x,y)}{2},\omega(d(x,y))\right]\subset \R,\end{equation}
    and $d_{\{x,y\},\omega}$ is the restriction of $|\cdot|$ to $S_{\{x,y\},\omega}$.
    We call $(S_{\{x,y\},\omega},d_{\{x,y\},\omega})$ the \emph{gapped segment associated to $\{x,y\}$ and $\omega$}. We regard $\{x,y\}$ as a subset of $S_{\{x,y\},\omega}$ identifying $x$ and $y$ with $0$ and $\omega(d(x,y))$ respectively. 
\end{definition} 

It is direct to check that, if $(S,\rho):=(S_{\{x,y\},\omega},d_{\{x,y\},\omega})$ is the gapped segment associated to a two point metric space $(\{x,y\},d)$ and any suitable distortion $\omega$, it holds that $S/\rho_{cf}=\{x,y\}$ and $\rho_{cf}(x,y)=d(x,y)$. 

\smallskip
We finish this subsection with the definition of the \emph{bending pseudometric}.

\begin{definition}[Bending]
Let $(M,\rho_M)$ be a pseudometric space, let $(x_i,y_i)_{i\in I}\subset M\times M$ be a set of pairs of points in $M$, and let $(a_i)_{i\in I}\subset [0,+\infty)$ be such that $\rho_M(x_i,y_i)\geq a_i$ for all $i\in I$. The \emph{bending} of $M$ by $\{(x_i,y_i,a_i)\}_{i\in I}$ is the pseudometric space $(M,\rho_B)$ where $\rho_B$ is the largest pseudometric in $M$ which is smaller than $\rho_M$ and such that $\rho_B(x_i,y_i)\leq a_i$ for all $i\in I$. 
\end{definition}

Let us point out a couple of easy facts about the previous definition.

\begin{remark}
\label{Remark:Bending_by_a_pair}
When bending a pseudometric space $(M,\rho_M)$ by a single triplet $(x,y,a)$, the resulting bending pseudometric $\rho_B$ can be explicitly written as
\begin{equation}\rho_B(p,q)=\min\{\rho_M(p,x)+\rho_M(y,q)+a,\rho_M(p,y)+\rho_M(q,x)+a,\rho_M(p,q)\}\end{equation}
for every $p,q\in M$.

\end{remark}

\begin{remark}
\label{Remark:Finite_bending}
Bending a pseudometric space by a finite set can be done sequentially: Let $(M,\rho_M)$ be a pseudometric space, let $(x_1,y_1),(x_2,y_2)\in M\times M$ and let $a_1,a_2\in [0,+\infty)$ such that $\rho_M(x_i,y_i)\geq a_i$ for $i=1,2$. If we denote by $\rho_B$ the bending of $(M,\rho_M)$ by the set $\{(x_i,y_i,a_i)\}_{i=1}^2$, then $\rho_B$ can also be obtained by bending $\rho_{B_1}$ by the single triplet $(x_2,y_2,a_2)$, where $\rho_{B_1}$ is the bending of $(M,\rho_M)$ by the single triplet $(x_1,y_1,a_1)$.
\end{remark}

\subsection{Preliminary results}

In this subsection, we prove technical results that we need in the sequel. All of these results aim to clarify the behavior of curve-flat functions and quotients with regards to attachment and bending. 

The first result we need is straightforward. 

\begin{lemma}[Finite bending preserves curve-flat functions]
\label{Lemma:Fin_ben_cf_functions}
    Let $(M,\rho)$ be a complete pseudometric space, and let $\rho_B$ be the bending pseudometric obtained bending $\rho$ by the finite set $\{(x_n,y_n,a_n)\}_{n\in F}$ where $x_n,y_n\in M$ and $\rho(x_n,y_n)\geq a_n$ for all $n\in F$. 

    If $f\colon (M,\rho_B)\rightarrow \R$ is a Lipschitz function such that $f\colon (M,\rho)\rightarrow \R$ is curve-flat, then $f$ is curve-flat for $\rho_B$ as well.
\end{lemma}

\begin{proof}

    Since bending by finitely many points can be done by sequentially bending by single pairs of points (Remark \ref{Remark:Finite_bending}), it is enough to show the result for $\rho_B$, where $\rho_B$ is the bending of $M$ by a single triplet $(x,y,a)$ with $\rho(x,y)\geq a$. 

    It follows from the explicit formula for bending by single pairs (Remark \ref{Remark:Bending_by_a_pair}) that for every $p\in M$ such that $\rho(p,\{x,y\})>0$, there exists $\varepsilon>0$ such that $\rho_B(p,q)=\rho(p,q)$ for all $q\in B_\rho(p,\varepsilon)$. 
    Let $K\subset \R$ be a compact set and let $\gamma\colon K\rightarrow (M,\rho_B)$ be a Lipschitz curve fragment. Define
   \begin{equation}
       K_1=\{t\in K\colon \rho(\gamma(t),\{x,y\})>0\}\subset K.
  \end{equation}    
    By the previous observation, it follows that for every $t\in K_1$ there exists $\varepsilon>0$ such that the restriction $\gamma_{|[t-\varepsilon,t+\varepsilon]}\colon K\rightarrow (M,\rho)$ is still Lipschitz. Hence, since $f$ is curve-flat for $\rho$, we get that $\lambda((f\circ\gamma)(K_1))=0$. 

    On the other hand, it is clear that $(f\circ\gamma)(K\setminus K_1)=f(\{x,y\})$, which has Lebesgue measure $0$ since it is finite. In conclusion, $\lambda((f\circ\gamma)(K))=0$, and thus $f$ is curve-flat.

\end{proof}

The next lemma follows as a direct consequence.

\begin{lemma}[Finite bending of a purely $1$-unrectifiable metric space]
\label{Lemma:Finite_bending_p1u}
    Let $(M,d_M)$ be a complete purely $1$-unrectifiable metric space. The bending of $(M,d_M)$ by a finite set is also complete and purely $1$-unrectifiable.
\end{lemma}

Additionally, Lemma \ref{Lemma:Fin_ben_cf_functions} allows us to describe with precision the curve-flat pseudometric associated to a single triplet bending.

\begin{lemma}[Finite bendings and curve-flat quotients commute]\label{Lemma:CF_of_pari_bending_formula}
    Let $(M,\rho)$ be a pseudometric space, let $x,y\in M$ with $x\neq y$, and let $a>0$ satisfy $a\leq \rho(x,y)$. Then 
    \begin{equation}(\rho_B)_{cf}=(\rho_{cf})_{B'},\end{equation}
    where 
    \begin{itemize}
        \item $\rho_B$ is the bending of $M$ by $(x,y,a)$
        \item $(\rho_{cf})_{B'}$ is the bending of $\rho_{cf}$ by $(x,y,b)$, where $b=\min\{\rho_{cf}(x,y),a\}$. 
    \end{itemize}
\end{lemma}
\begin{proof}
    Note that $(\rho_{cf})_{B'}\leq \rho_B$. Indeed, $\rho_B$ is the largest pseudometric in $M$ smaller than $\rho$ and such that $\rho(x,y)\leq a$. Since $(\rho_{cf})_{B'}$ also satisfies these conditions, the claim follows. As a consequence, the function $g_p:(M,\rho_B)\rightarrow \mathbb{R}$ defined by $g_p(q)=(\rho_{cf})_{B'}(p,q)$ is 1-Lipschitz for each $p\in M$. By Lemma \ref{Lemma:CF_pseudometric_is_CF}, the function $g_p$ is curve-flat for $(M,\rho)$, and thus by Lemma \ref{Lemma:Fin_ben_cf_functions} we conclude that $g_p$ is curve-flat for $\rho_B$. Since this holds for all $p\in M$, we get $(\rho_B)_{cf}\geq (\rho_{cf})_{B'}$.
    
    On the other hand, the inequality $(\rho_B)_{cf}\leq(\rho_{cf})_{B'}$ follows from the definition of bending by noting that $(\rho_B)_{cf}\leq \rho_{cf}$ and $(\rho_B)_{cf}(x,y)\leq b$.
\end{proof}

The next result is the last of the technical tools regarding bending and attachment that we need. It can be interpreted as saying that, for purely $1$-unrectifiable frames, the attachment and curve-flat pseudometrics commute (after necessary adjustment through bending). 

\begin{lemma}[Attachment and curve-flat pseudometrics commute]
\label{Lemma:Attach_and_CF_commute}
    Let $(M,d_M)$ be a complete purely $1$-unrectifiable metric space. For each $n\in\N$, let $(S_n,\rho_n)$ be a compact pseudometric space such that:
    \begin{itemize}
        \item $S_n\cap S_m\subset M$ for all $n\neq m$.
        \item The intersection $S_n\cap M$ is a $2$-point set $\{x_n,y_n\}$ for all $n\in\N$, where the metric $d_M$ and the pseudometric $\rho_n$ coincide.
    \end{itemize}
    Write $Y:=M\cup\left(\bigcup_nS_n\right)$ and  $\rho_M$ the pseudometric obtained by attachment of the threads $(S_n,\rho_n)_n$ into the frame $(M,d_M)$. 

    Then the curve-flat pseudometric $(\rho_Y)_{cf}$ on $Y$ coincides with the pseudometric obtained by attachment of the threads $(S_n,\rho_{B_n})_n$ into the frame $(M,\rho_{B})$, where 
    \begin{itemize}
        \item $(M,\rho_{B})$ is the bending of $(M,d_M)$ by $\{(x_n,y_n,(\rho_n)_{cf}(x_n,y_n))\}$.
    \item $(S_n,\rho_{B_n})$ is the bending of $(S_n,(\rho_n)_{cf})$ by $\{(x_n,y_n,\rho_B(x_n,y_n))\}$
    \end{itemize}
    
\end{lemma}

\begin{proof}
  We split the proof into three parts. 
  
\smallskip
	\noindent
	\textbf{Part 1: Set up.}
    
     Write $a_n:=(\rho_n)_{cf}(x_n,y_n)$ and $b_n:=\rho_B(x_n,y_n)$ for all $n\in\N$. First of all, in order to properly define the pseudometric of the attachment of $(M,\rho_B)$ with the threads $(S_n,\rho_{B_n})_n$, we must check that $\rho_B$ and $\rho_{B_n}$ coincide in $S_n\cap M=\{x_n,y_n\}$ for every $n\in\N$. However, this clearly holds since $\rho_{B_n}$ is the pseudometric in $S_n$ defined by bending $(S_n,\rho_n)$ by the single triplet $(x_n,y_n,b_n)$. 
        
     Let us write $\rho_*$ as the pseudometric defined by the attachment of $(M,\rho_B)$ with $(S_n,\rho_{B_n})_n$, that is, $\rho_*$ is the largest pseudometric in $Y$ which coincides with $\rho_B$ in $M$ and coincides with $\rho_{B_n}$ in $S_n$ for every $n\in\N$.

The goal is to show that $(\rho_Y)_{cf}=\rho_*$.

\smallskip
	\noindent
	\textbf{Part 2: $(\rho_Y)_{cf}\geq \rho_*$}      
    
    By definition of the curve-flat pseudometric, we will be done if we show that for any $p\in Y$, the function $g_p\colon (Y,\rho_Y)\rightarrow \R$ given by $g_p(q):=\rho_*(p,q)$ for all $q\in Y$ is $1$-Lipschitz and curve-flat for the pseudometric $\rho_Y$.

        First of all, notice that both $\rho_Y$ and $\rho_*$ are attachment pseudometrics with the same frame and threads, where all pseudometrics used to define $\rho_*$ are smaller than those used to define $\rho_Y$. Hence $\rho_*\leq \rho_Y$ and it follows that $g_p$ is $1$-Lipschitz for $\rho_Y$.

        Next, we show that the restriction $(g_p)_{|S_n}\colon (S_n,\rho_n)\rightarrow \R$ is curve-flat for all $n\in \N$. Fix $n\in \N$. The pseudometric $\rho_*$ restricted to $S_n$ coincides with the pseudometric $\rho_{B_n}$. Therefore, the map $(g_p)_{|S_n}\colon (S_n,\rho_{B_n})\rightarrow \R$ is $1$-Lipschitz. Since $\rho_{B_n}$ is obtained by bending $(S_n,(\rho_n)_{cf})$, it holds that $\rho_{B_n}\leq (\rho_n)_{cf}$ and hence $(g_p)_{|S_n}\colon (S_n,(\rho_n)_{cf})\rightarrow \R$ is also $1$-Lipschitz. By Lemma \ref{Lemma:CF_pseudometric_is_CF}, we get that $(g_p)_{|S_n}$ is curve-flat with respect to the pseudometric $\rho_n$, as desired. 

        To finish this part, let $\gamma\colon K\rightarrow (Y,\rho_Y)$ be a Lipschitz curve fragment. For every $n\in\N$, denote $K_n:=\gamma^{-1}(S_n)$ and $K_M:=\gamma^{-1}(M)$. Clearly $K=K_M\cup \bigcup_{n\in \N}K_n$ and thus 
        \begin{equation}
            \left(g_p\circ\gamma\right)(K)=(g_p\circ\gamma)(K_M)\cup \bigcup_{n\in\N}(g_p\circ\gamma)(K_n).
        \end{equation}
        Since $M$ is purely $1$-unrectifiable, it is immediate that $(g_p\circ\gamma)(K_M)$ has Lebesgue measure $0$. Additionally, the same holds for $(g_p\circ\gamma)(K_n)$ by considering the restriction $\gamma_{|K_n}$, since we have shown that $(g_p)_{|S_n}$ is curve-flat for the pseudometric $\rho_n$. Therefore $\lambda((g_p\circ\gamma)(K))=0$ and thus $g_p$ is curve-flat, as claimed.

\smallskip
	\noindent
	\textbf{Part 3: $\rho_*\geq (\rho_M)_{cf}$} 
    
    We will prove this by showing that $(\rho_Y)_{cf}$ coincides with $\rho_B$ in $M$ and with $\rho_{B_n}$ in $S_n$. This will indeed prove the desired inequality since, by definition, the attachment pseudometric $\rho_*$ is the biggest pseudometric in $Y=M\cup\left(\bigcup_{n\in\N}S_n\right)$ satisfying these conditions. In fact, using Part 2, we only need to show that $ \rho_{B_n}\geq(\rho_Y)_{cf}$ and $\rho_B\geq(\rho_Y)_{cf}$ in each $S_n$ and in $M$ respectively.
        
        First, we prove that $\rho_B\geq(\rho_Y)_{cf}$ in $M$. The pseudometric $\rho_B$ is, by definition, the largest pseudometric in $M$ which is smaller than $d_M$ and such that $\rho_B(x_n,y_n)\leq a_n$ for all $n\in\N$, so we only need to check that the restriction of $(\rho_Y)_{cf}$ to $X$ also satisfies these conditions. 

        Indeed, $(\rho_Y)_{cf}$ is smaller than $d_M$ in $M$, since it is smaller than $\rho_Y$ which coincides with $d_M$ in $M$. Secondly, for every $n\in\N$, we have that $(\rho_Y)_{cf}(x_n,y_n)\leq a_n$. In fact, it holds that $(\rho_Y)_{cf}\leq (\rho_n)_{cf}$ in $S_n$. Certainly, for any $p,q\in S_n$ we have that 
        \begin{equation}
            (\rho_Y)_{cf}(p,q)=\sup\{|f(p)-f(q)|\colon f\text{ is }1\text{-Lipschitz and }f\in CF(Y,\rho_Y)\},
        \end{equation}
        so it suffices to observe that $|f(p)-f(q)|\leq (\rho_n)_{cf}(p,q)$ for all $f\colon (Y,\rho_Y)\rightarrow \R$ curve-flat $1$-Lipschitz functions. This holds because the restriction $f_{|S_n}\colon (S_n,\rho_n)\rightarrow \R$ of such a function is also $1$-Lipschitz and curve-flat.

        Next, we show that for every $n\in\N$, $\rho_{B_n}$ is larger than $(\rho_Y)_{cf}$ in $S_n$. We use a similar argument: the pseudometric $\rho_{B_n}$ is by definition the largest metric on $S_n$ which is smaller than $(\rho_n)_{cf}$ and such that $\rho_{B_n}(x_n,y_n)\leq b_n$. We clearly have that $(\rho_Y)_{cf}(x_n,y_n)\leq b_n$ (in fact, $(\rho_Y)_{cf}(x_n,y_n)=b_n$), since $(\rho_Y)_{cf}$ coincides with $\rho_B$ in $M$. Moreover, we have already shown above that $(\rho_Y)_{cf}\leq (\rho_n)_{cf}$ in $S_n$. We conclude that $\rho_{B_n}\geq(\rho_Y)_{cf}$ in $S_n$, as desired.

\end{proof}

\subsection{Main construction}\label{section:Main_Construction}

\subsubsection{Assumptions on parameters.}
Our construction depends on several parameters, which are not unique for a given triplet $(M,\alpha,\varepsilon)$. We say that a tuple $\Gamma=(\omega,(\alpha_n,x_n,y_n)_n)$ is \emph{admissible for $(M,\alpha,\varepsilon)$}, if all the following conditions are met:

\begin{enumerate}[label=(a.\arabic*)]
    \item $\omega:[0,\infty)\rightarrow [0,\infty)$ is a local distortion that satisfies $\omega(\diam(M))= \diam(M)$. By concavity, this implies that $\omega(t)\geq t$ for all $t\in [0,\diam(M)]$, and thus $\omega\circ d\geq d$ in $M$. \label{adm:distortion}
   
    \item $(\alpha_n)_n$ is an increasing sequence of ordinals less than $\alpha$ such that for every $\beta<\alpha$ there exists $N\in\N$ so that $\alpha_n\geq \beta$ for all $n>N$. \label{adm:ordinal_approach_alpha}
    
    \item $(x_n)_n$ and $(y_n)_n$ are sequences of points in $M$ such that $x_n\neq y_n$ for every $n\in\N$ and for every $p\neq q\in M$ with $d(p,q)\leq 2^{-k}\diam(M)$, $k\in \N\cup\{0\}$, there exists $n\in\N$ such that 
    \begin{equation}
        2\left((\omega\circ d)(p,x_n)+(\omega\circ d)(q,y_n)\right)\leq 2^{-k}\varepsilon d(p,q).
    \end{equation}\label{adm:x_ny_n_pairs}
    or
    \begin{equation}
        2\left((\omega\circ d)(p,y_n)+(\omega\circ d)(q,x_n)\right)\leq 2^{-k}\varepsilon d(p,q).
\end{equation}

    \item $(d(x_n,y_n))_n$ goes to $0$.\label{adm:d(xnyn)_goes_to_zero}

\end{enumerate}

\smallskip
It is not obvious that an admissible tuple exists for a given $(M,\alpha,\varepsilon)$. In fact, condition \ref{adm:d(xnyn)_goes_to_zero} cannot hold in finite metric spaces. However, this is the only obstacle for the existence of an admissible tuple.
\begin{lemma}\label{lemma:exists_admissible}
    Given an infinite compact metric space $(M,d)$, an $\varepsilon>0$ and a countable ordinal $\alpha$, there exists an admissible tuple for $(M,\alpha, \varepsilon)$.
\end{lemma}
\begin{proof}
It is clear that there exists a local distortion $\omega\colon [0,+\infty)\rightarrow[0,+\infty)$ such that $\omega(\diam(M))=\diam(M)$ (this can be done, for instance, by scaling a snowflake distortion $\omega(t)=t^\beta$, $\beta<1$). Similarly, it is easy to construct a countable sequence of ordinals satisfying \ref{adm:ordinal_approach_alpha}: If $\alpha$ is a successor ordinal, it suffices to choose $\alpha_n=\alpha-1$ for all $n\in\N$; and if $\alpha$ is a limit ordinal, we only need that $\sup_{n\in\N}\alpha_n=\alpha$. Note that if $\alpha=0$, the condition is trivially satisfied.

It only remains to define the sequences $(x_n)_n$ and $(y_n)_n$ satisfying \ref{adm:x_ny_n_pairs} and \ref{adm:d(xnyn)_goes_to_zero}. We may assume that $\varepsilon<1$ without loss of generality. Let $\delta_n:=2^{-(n-1)}\diam(M)$ and let $\varepsilon_n:=\frac{2^{-n}\delta_{n+1}}{4}\varepsilon$ for every $n\in\N$. Choose, for every $n\in \N$, a finite set $E_n\subset M$ that is $\varepsilon_n$-dense in $(M,\omega\circ d)$, and define 
\begin{equation}
    F_n:=\{\{x,y\}\colon x\neq y\in E_n, \text{ and }d(x,y)\leq \delta_n+2\varepsilon_n\},
\end{equation}
which is also finite. We claim that for every $k\in \N\cup\{0\}$ and every $p\neq q\in M$ with $2^{-(k+1)}\diam(M)\leq d(p,q)\leq 2^{-k}\diam(M)$ there exist $\{x,y\}\in F_{k+1}$ such that 
\begin{equation}
    2\left((\omega\circ d)(p,x)+(\omega\circ d)(q,y)\right)\leq 2^{-k}\varepsilon d(p,q).
\end{equation}
Fix $k\in \N\cup\{0\}$ and $p\neq q \in M$ with $d(p,q)\leq 2^{-k}\diam (M)$. We have that $\delta_{k+2}\leq d(p,q)\leq \delta_{k+1}$. Since $E_{k+1}$ is $\varepsilon_{k+1}$-dense in $(M,\omega\circ d)$, there exist $x,y\in E_{k+1}$ such that 
    \begin{equation}
        (\omega\circ d)(x,p)\leq \varepsilon_{k+1}\qquad\text{and}\qquad(\omega\circ d)(y,q)\leq \varepsilon_{k+1}.
    \end{equation}
    The triangle inequality and the fact that $d\leq \omega\circ d$ imply that $d(x,y)\leq \delta_{k+1}+2\varepsilon_{k+1}$. On the other hand, reverse triangle inequality shows 
    \begin{align}
    (\omega\circ d)(x,y)&\geq \delta_{k+2}-2\varepsilon_{k+1}=\delta_{k+2}\left(1-\frac{2^{-k+1}}{2}\varepsilon\right)>0,    
    \end{align}
    where we used that $\varepsilon<1$. 
    
    Therefore, $x\neq y$ and $\{x,y\}\in F_{k+1}$. Finally, observe that
    \begin{align}
        \left((\omega\circ d)(p,x)+(\omega\circ d)(q,y)\right)&\leq 2\varepsilon_{k+1}\leq 2\frac{2^{-k}\delta_{k+2}}{4}\varepsilon\leq\frac{1}{2}2^{-k}\varepsilon d(p,q),
    \end{align}
    as desired. 

    To finish the proof, notice that $\bigcup_{n\in\N}F_n$ is a countable set. Since $M$ is infinite and compact, $F_n$ is nonempty for infinitely many $n\in\N$, so there exist sequences $(x_n)_n,(y_n)_n\subset M$ such that $\bigcup_{n\in\N}F_n=\bigcup_{n\in\N}\{x_n,y_n\}$. By the previous discussion, condition \ref{adm:x_ny_n_pairs} is satisfied, and since each $F_n$ is finite and $d(x,y)\leq \delta_n+2\varepsilon_n$ for all $\{x,y\}\in F_n$, condition \ref{adm:d(xnyn)_goes_to_zero} also holds.

\end{proof}

\begin{remark}
    Note that in the admissibility conditions regarding the sequence $(\alpha_n,x_n,y_n)_n$, only the tail of the sequence matters, in the sense that if we concatenate a finite sequence of the form $(1,p_i,q_i)_{i=1}^N$ before it, the resulting tuple will be admissible if the tuple $\Gamma=(\omega,(\alpha_n,x_n,y_n)_n)$ is admissible. 
\end{remark}

The admissibility condition \ref{adm:x_ny_n_pairs} on the pairs $(x_n,y_n)_n$ implies that these pairs are ``dense enough" in the metric space $M$ even though, by \ref{adm:d(xnyn)_goes_to_zero} $(d(x_n,y_n))_n$ goes to $0$. This is more concretely expressed in the following lemma.

\begin{lemma}
\label{Lemma:Full_bending_isom_better_small_scales}
    Let $\Gamma=(\omega,(\alpha_n,x_n,y_n)_n)$ be an admissible tuple for $(M,\alpha,\varepsilon)$. Let $\rho$ be any metric defined on $M$ satisfying $d\leq \rho\leq \omega\circ d$. Denote by $\rho_B$ the bending pseudometric obtained bending $\rho$ by the set of all triplets $(x_n,y_n,d(x_n,y_n))_n$. 
.

    Then, for every $k\in\mathbb{N}\cup \{0\}$ and every $p\neq q\in M$ with $d(p,q)\leq 2^{-k}\diam(M)$ it holds that 
    \begin{equation}
    \label{eq:isom_better_small_scales}
    d(p,q)\leq \rho_B(p,q)\leq (1+2^{-k}\varepsilon)d(p,q).
    \end{equation}
\end{lemma}
\begin{proof}
    Recall that $\rho_B$ is, by definition, the biggest pseudometric in $M$ such that $\rho_B\leq \rho$ and such that $\rho_B(x_n,y_n)\leq d(x_n,y_n)$ for every $n\in\N$. Therefore, it follows directly that $d\leq \rho_B$. It only remains to show the second inequality of equation \eqref{eq:isom_better_small_scales}.

    Let $k\in\mathbb{N}\cup\{0\}$ and let $p\neq q\in M$ with $d(p,q)\leq 2^{-k}\varepsilon$. By admissibility condition \ref{adm:x_ny_n_pairs}, and changing $p$ and $q$ if necessary, there exists $n\in\N$ such that 
    \begin{equation}
        2\left((\omega\circ d)(p,x_n)+(\omega\circ d)(q,y_n)\right)\leq (2^{-k}\varepsilon )d(p,q).
    \end{equation}
    By the bending definition, $\rho_B(x_n,y_n)\leq d(x_n,y_n)$. Therefore, we have
    \begin{align}
        \rho_B(p,q)&\leq \rho_B(p,x_n)+\rho_B(x_n,y_n)+\rho_B(y_n,q)\\
        &\leq \rho(p,x_n)+d(x_n,y_n)+\rho(y_n,q) \nonumber\\
        &\leq 2\left((\omega\circ d)(p,x_n)+(\omega\circ d)(q,y_n)\right)+d(p,q) \nonumber\\
        &\leq (1+2^{-k}\varepsilon)d(p,q).\nonumber \qedhere
    \end{align}

\end{proof}

Note that, under the assumptions of the previous lemma, the identity map $\iota:(M,\rho_B)\rightarrow (M,d)$ is a $(1+\varepsilon)$-isometry. The Lemma shows in fact a stronger statement, namely, that $\iota$ gets closer and closer to being an isometry at smaller scales. This implies by an elementary argument that it is a \textit{true} isometry for geodesic compact spaces $(M,d)$. We record this fact (in the language of pseudometrics) and its simple proof for later reference.

\begin{lemma}
\label{Lemma:almost_isom_small_scales_geodesic}
    Let $(M,d_M)$ be a length metric space, let $(N,\rho_N)$ be a pseudometric space, and let $(r_n)_n\subset (0,+\infty)$ be a sequence converging to $0$ with $r_1=\diam(M)$. Let $\varphi\colon (M,d_M)\rightarrow (N,\rho_N)$ be a map such that for every $n\in\N$ and for every pair of points $p\neq q\in M$ with $d_M(p,q)\leq r_n$ it holds that 
    \begin{equation}
    \label{eq:isom_small_scales_rn}
        d_M(p,q)\leq \rho_N(\varphi(p),\varphi(q))\leq (1+r_n)d_M(p,q).
    \end{equation}
    Then $d_M(p,q)=\rho_N(\varphi(p),\varphi(q))$ for all $p,q\in M$.
\end{lemma}
\begin{proof}
    Fix $p\neq q\in M$. Since $r_1=\diam(M)$, it follows that $d_M(p,q)\leq \rho_N(\varphi(p),\varphi(q))$. To prove the reverse inequality $\rho_N(\varphi(p),\varphi(q))\leq d_M(p,q)$, it suffices to show that $\rho_N(\varphi(p),\varphi(q))\leq (1+r_n)^2d_M(p,q)$ for every $n\in\N$.  

    Since $M$ is length, given $n\in\N$ there exists a finite sequence $(a_i)_{i=0}^k\subset M$ with $a_0=p$ and $a_k=q$ such that $d_M(a_i,a_{i+1})\leq r_n$ and such that $\sum_{i=0}^{k-1} d_M(a_i,a_{i+1})\leq(1+r_n)d_M(p,q)$. By assumption, we have that 
\begin{align}
    \rho_N(\varphi(p),\varphi(q))&\leq\sum_{i=0}^{k-1}d_N(\varphi(a_i),\varphi(a_{i+1}))\\
    &\leq \sum_{i=0}^{k-1}(1+r_n)d_M(a_i,a_{i+1})=(1+r_n)^2d_M(p,q),\nonumber
\end{align}
as desired.
\end{proof}

\subsubsection{Proof of Main Theorem}
\begin{theorem}
\label{Theorem:main_theorem_2_general}
For every compact metric space $(M,d)$, every  $\varepsilon>0$, and every countable ordinal $\alpha$, there exists a compact metric space $(Y,\rho)$ with $\alpha_{cf}(Y)= \alpha+\alpha_{cf}(M)$ so that the following conditions hold.
    \begin{enumerate}
    \item The diameter of $(Y,\rho)$ is less than $3\diam (M,d)$.
    \item There exists an injective map $\iota\colon M\rightarrow Y$ such that $\iota(M)$ is total in $(Y,\rho_{cf}^\alpha)$ and
    \begin{equation}d(x,y)\leq \rho^\alpha_{cf}(\iota(x),\iota(y))\leq(1+\varepsilon)d(x,y)\end{equation}
    for all $x,y\in M$. Moreover, if $x,y\in M$ satisfy that $d(x,y)=\diam(M,d)$, then $d(x,y)=\rho(\iota(x),\iota(y))$.
    
    \item If $M$ is a geodesic space, $M$ is a gapped segment, or $M$ is purely 1-unrectifiable, then $\iota$ satisfies 
     \begin{equation}d(x,y)= \rho^\alpha_{cf}(\iota(x),\iota(y))\end{equation} for all $x,y\in M$.

    \end{enumerate}
\end{theorem}
Let us first provide a sketch of the proof.  Let $(M,d_M)$ be a compact metric space, let $\alpha$ be a countable ordinal, and let $\varepsilon>0$. Let $\Gamma=(\omega,(\alpha_n,x_n,y_n)_n)$ be  $(M,\alpha,\varepsilon)$-admissible. We construct a compact metric space $(Y,\rho)$ in the following way. 

For every $n\in \N$, we define a metric space  $(S_n,\rho_n)$ such that 
\begin{enumerate}[label=(b.\arabic*)]
\item\label{sketch:isometric} $(S_n,(\rho_n)_{cf}^{\alpha_n})$ is isometric to the gapped segment 
    \begin{equation}(S_{\{x_n,y_n\},\omega},d_{\{x_n,y_n\},\omega});\end{equation}
    \item $(\rho_n)_{cf}^{\beta}(x_n,y_n)=(\omega\circ d)(x_n,y_n)$ for all $\beta<\alpha$.\label{sketch:big}
\end{enumerate}

Next, we define $(Y,\rho)$ as the attachment of the threads $(S_n,\rho_n)_n$ into a frame $(M,\omega\circ d)$ anchoring at pairs of points $(x_n,y_n)$. The spaces $(S_n,\rho_n)_n$ are taken to be small enough to keep the resulting space $(Y,\rho)$ compact.
We prove that if we apply the curve-flat quotient operation to $(Y,\rho)$ of any order $\beta$ strictly less than $\alpha$, then the frame $(M,d)$ stays purely 1-unrectifiable thanks to \ref{sketch:big}, bent only at finitely many pairs of points where the threads $(S_n,(\rho_n)_{cf}^{\beta})$ have collapsed into pairs. Only when we apply the curve-flat quotient operation of order exactly $\alpha$,  all threads $(S_n,(\rho_n)_{cf}^{\alpha_n})_n$ collapse into the pairs $(\{x_n,y_n\},d_M)$ by \ref{sketch:isometric}. After the collapse of all threads, we recover the original metric space $(M,d_M)$ by Lemma \ref{Lemma:Full_bending_isom_better_small_scales}, with small distortion.

We also point out that the statement of the main theorem is more technical than Theorem \ref{Main_Theorem_4_compacthighordercf}. The additional conditions we include are useful when applying the inductive hypothesis, and Theorem \ref{Main_Theorem_4_compacthighordercf} clearly follows as a corollary.

\begin{proof}[Proof of Theorem \ref{Theorem:main_theorem_2_general}]

We first assume that $M$ is not purely 1-unrectifiable and thus infinite. We prove this result by induction on $\alpha$. It is trivially true for $\alpha=0$ with $(Y,\rho)=(M,d)$. Let $\alpha>0$ and assume the result holds for all $\beta<\alpha$.

By Lemma \ref{lemma:exists_admissible}, there exists a $(M,\alpha,\varepsilon)$-admissible tuple $\Gamma=(\omega, (\alpha_n,x_n,y_n)_n)$.
Note that by admissibility condition \ref{adm:distortion}, it holds that $(M,\omega\circ d)$ is compact and its diameter is equal to the diameter of $M$.

The metric space $(Y,\rho)$ is built by attachment, and its frame is $(M,\omega\circ d)$. In order to define its threads we will use the following claim.

\begin{claim}

\label{Claim1}
    There exist a sequence of compact metric spaces $\{(S_n,\rho_n)\}_n$ and $n_0\in \N$ with the following properties:
    \begin{enumerate}[label=(S.\arabic*)]
    \item For all $n\in \N$,
    \begin{equation}S_n\cap M=\{x_n,y_n\};\end{equation} 
   and the pseudometric $\rho_n$ coincides with $\omega\circ d$ in $\{x_n,y_n\}$.
    \item For all $n_1,n_2\in \N$  with $n_1\neq n_2$, we have \begin{equation}S_{n_1}\cap S_{n_2}\subset M;\end{equation}
    \item The sequence $(\diam(S_n,\rho_n))_n$ is bounded by $\diam(M,d)$ and converges to $0$.

     \item For  $n\in\N$ and $\alpha_n+1\leq\beta\leq \alpha$, the pair $\{x_n,y_n\}$ is total in $(S_n,(\rho_n)_{cf}^\beta)$ and
        \begin{equation}
        \label{eq:S_k_collapse_by_ord_alpha}
            d(x_n,y_n)= (\rho_n)^{\beta}_{cf}(x_n,y_n);
        \end{equation}
    
    \item If $n\geq n_0$, then the space $(S_n,(\rho_n)_{cf}^{\alpha_n})$ is isometric to the gapped segment 
    \begin{equation}(S_{\{x_n,y_n\},\omega},d_{\{x_n,y_n\},\omega});\end{equation}

        \item  If $n\geq n_0$, then
        \begin{equation}
        \label{eq:S_k_dont_collapse_b4_ord_alpha}
            (\rho_n)^{\beta}_{cf}(x_n,y_n)=(\omega\circ d)(x_n,y_n)\qquad\text{for all $\beta\leq \alpha_n$. }
        \end{equation}

    \end{enumerate}
\end{claim}

\begin{proof}[Proof of Claim \ref{Claim1}]

Find $n_0\in \N$ so that $d(x_n,y_n)\leq 3\diam(M,\omega\circ d)$ for all $n\geq n_0$. For $n<n_0$, we define $(S_n,\rho_n)$ as the gapped segment $(S_{\{x_n,y_n\},\omega},d_{\{x_n,y_n\},\omega})$ such that $S_n\cap M=\{x_n,y_n\}$ and $S_{n_1}\cap S_{n_2}\subset M$. The second part of condition (S.1) is satisfied, and $\diam(S_n,\rho_N)\leq \diam(M,d)$ since $\omega(\diam(M,d))=\diam(M,d)$. Additionally, since the first order curve-flat quotient of $(S_n,\rho_n)$ is isometric to $(\{x_n,y_n\},d)$, condition (S.4) also holds. 

For $n\geq n_0$, we use the inductive hypothesis for the gapped segment $(S_{\{x_n,y_n\},\omega},d_{\{x_n,y_n\},\omega})$, the ordinal $\alpha_n$ and any $\varepsilon$, in order to find a metric space $(S_n,\rho_n)$ with $\diam(S_n,\rho_n)\leq 3(\omega(d(x_n,y_n))$ and an injective map $\iota_n\colon S_{\{x_n,y_n\},\omega}\rightarrow S_n$ such that $\iota_n(S_{\{x_n,y_n\},\omega})$ is total in $(S_n,(\rho_n)^{\alpha_n}_{cf})$ and 
\begin{equation}
\label{eq:Sn_is_high_order_gapped_segment}
    d_{\{x_n,y_n\},\omega}(x,y)= (\rho_n)_{cf}^{\alpha_n}(\iota_n(x),\iota_n(y))
\end{equation}
for all $x,y\in S_{\{x_n,y_n\},\omega}$. Moreover, it is satisfied that $d_{\{x_n,y_n\},\omega}(x,y)=\rho_n(\iota_n(x),\iota_n(y))$ for $x,y\in S_{\{x_n,y_n\},\omega}$ such that $d_{\{x_n,y_n\},\omega}(x,y)=\diam(S_{\{x_n,y_n\},\omega},d_{\{x_n,y_n\},\omega})=\omega(d(x_n,y_n))$. In particular, it holds that 
\begin{equation}
\label{eq:Sn_is_attachable}
    (\omega\circ d)(x_n,y_n)=\rho_n(\iota_n(x_n),\iota_n(y_n)).
\end{equation}

Let us identify $S_{\{x_n,y_n\},\omega}$ with its image through $\iota_n$. Then $\{x_n,y_n\}\subset S_n$ and conditions (S.1) and (S.2) are satisfied (we use equation \eqref{eq:Sn_is_attachable} for the second part of condition (S.1)). Condition (S.3) also holds, since $\diam(S_n,\rho_n)\leq 3(\omega(d(x_n,y_n))\leq \diam(M,d)$, and by admissibility condition \ref{adm:d(xnyn)_goes_to_zero} the sequence $(3(\omega(d(x_n,y_n))_n$ goes to $0$.

The condition (S.5) holds by the induction hypothesis. We thus have that $(S_n,(\rho_n)_{cf}^{\alpha_n+1})$ is isometric to the pair $(\{x_n,y_n\},d)$. Equivalently, the pair $\{x_n,y_n\}$ is total in $(S_n,(\rho_n)_{cf}^{\alpha_n+1})$ and 
\begin{equation}
    \label{eq:Sn_cf_alphan+1_is_pair_xnyn}
    d(x_n,y_n)= (\rho_n)_{cf}^{\alpha_n+1}(x_n,y_n).
\end{equation}
Since purely $1$-unrectifiable spaces are isometric to their curve-flat quotients, the previous statement holds for all ordinals $\beta$ satisfying $\alpha_n+1\leq \beta\leq\alpha$. This proves (S.4).

Finally, to prove (S.6), fix $\beta\leq \alpha_n$ and observe that
\begin{align}
    (\rho_n)_{cf}^\beta(x_n,y_n)&\geq (\rho_n)_{cf}^{\alpha_{n}}(x_n,y_n)\geq (\omega\circ d)(x_n,y_n)\\
    &=\rho_n(x_n,y_n)\geq (\rho_n)_{cf}^\beta(x_n,y_n),\nonumber
\end{align}
where we used equations \eqref{eq:Sn_is_high_order_gapped_segment} and \eqref{eq:Sn_is_attachable}.

\end{proof}

\textit{Proof of Theorem \ref{Theorem:main_theorem_2_general} continues.}

Define $Y=M\cup\bigcup_{n\in\N}S_n$, and the metric $\rho$ is defined by attachment of the frame $(M,\omega\circ d)$ with the frames $(S_n,\rho_n)$. Conditions (S.1) and (S.2) show that the attachment is well defined. 

To show that $(Y,\rho)$ is compact, we prove it contains an $r$-dense compact subset for every $r>0$. Given $r>0$, by condition (S.3) there exists $n_r\in\N$ such that $\diam(S_n,\rho_n)<r$ for all $n\geq n_r$. This implies that the compact metric space $M\cup\bigcup_{n\leq n_r}S_n$ is $r$-dense in $(Y,\rho)$.

Since the diameter of $(M,w\circ d)$ and $(S_n,\rho_n)$ is less than $\diam (M,d)$ for all $n\in\N$, an easy computation shows that $\diam(Y,\rho)\leq 3\diam(M,d)$, so (1) is proven.

For (2), the injective map $\iota\colon M\rightarrow Y$ we consider is simply the inclusion map. We need to prove that $M$ is total in $(Y,\rho^\alpha_{cf})$ and that 
\begin{equation}
    d(x,y)\leq \rho_{cf}^\alpha(x,y)\leq (1+\varepsilon)d(x,y)
\end{equation}
for all $x,y\in M$.

To do so, we will prove the following statement.

\begin{claim}
\label{Claim2}
For every $0\leq \beta\leq \alpha$, the pseudometric $\rho_{cf}^\beta$ is the pseudometric in $Y$ obtained by attachment of the threads $(S_n,\eta_n^\beta)_n$ into the frame $(M,d_\beta)$, where
\begin{itemize}
    \item $(M,d_\beta)$ is the bending of $(M,\omega\circ d)$ by the set $\{(x_n,y_n,(\rho_n)_{cf}^\beta(x_n,y_n))\}_n$
    \item $(S_n,\eta_n^\beta)$ is the bending of $(S_n,(\rho_n)_{cf}^\beta)$ by $\{(x_n,y_n,d_\beta(x_n,y_n))\}$
\end{itemize}
\end{claim}
\begin{proof}[Proof of Claim \ref{Claim2}]
    We prove this claim by induction on $\beta$. 

    First we deal with the successor ordinal case. Assume we have shown the claim for an ordinal $\beta<\alpha$. We will prove that $\rho_{cf}^{\beta+1}$ is the pseudometric in $Y$ obtained by attachment of the threads $(S_n,\eta_n^{\beta+1})_n$ into the frame $(M,d_{\beta+1})$, where
    \begin{itemize}
    \item $(M,d_{\beta+1})$ is the bending of $(M,\omega\circ d)$ by the set $\{(x_n,y_n,(\rho_n)_{cf}^{\beta+1}(x_n,y_n))\}_n$;
    \item $(S_n,\eta_n^{\beta+1})$ is the bending of $(S_n,(\rho_n)_{cf}^{\beta+1})$ by $\{(x_n,y_n,d_{\beta+1}(x_n,y_n))\}$.
\end{itemize}

    We know that $\rho_{cf}^{\beta+1}$ is the first order curve-flat pseudometric associated to $\rho_{cf}^\beta$, which, by inductive hypothesis, is the attachment of the threads $(S_n,\eta_n^\beta)_n$ into the frame $(M,d_\beta)$, as defined in the statement of the claim. We start by verifying that the hypothesis of Lemma \ref{Lemma:Attach_and_CF_commute} are satisfied for the attachment pseudometric $\rho_{cf}^\beta$. First, note that, by (S.5) in Claim \ref{Claim1}, $(\rho_n)_{cf}^\beta(x_n,y_n)=(\omega\circ d)(x_n,y_n)$ for all $n\geq n_0$ such that $\alpha_n\geq \beta$. Therefore, using admissibility condition \ref{adm:ordinal_approach_alpha}, we get that the metric $d_\beta$ on $(M,\omega\circ d)$ is in fact a bending metric defined by a finite set of triplets. Hence, Lemma \ref{Lemma:Finite_bending_p1u} implies that $(M,d_\beta)$ is purely $1$-unrectifiable. Secondly, we also need to verify that $d_\beta$ and $\eta_n^\beta$ coincide in the set $S_n\cap M$. However, this is immediate by the definition of $\eta_n^\beta$.

    We can use now Lemma \ref{Lemma:Attach_and_CF_commute} to get that $\rho_{cf}^{\beta+1}$ is the attachment pseudometric on $Y$ obtained by attachment of the threads $(S_n,\rho_{B_n})_n$ into the frame $(M,\rho_B)$ where

     \begin{itemize}
        \item $(M,\rho_{B})$ is the bending of $(M,d_\beta)$ by $\{(x_n,y_n,(\eta_n^\beta)_{cf}(x_n,y_n))\}_n$;
    \item $(S_n,\rho_{B_n})$ is the bending of $(S_n,(\eta_n^\beta)_{cf})$ by $\{(x_n,y_n,\rho_B(x_n,y_n))\}$.
    \end{itemize}
    We must show that $\rho_B=d_{\beta+1}$ and $\rho_{B_n}=\eta_n^{\beta+1}$ for all $n\in\N$. 
    
    We start with $\rho_B=d_{\beta+1}$, first showing that $\rho_B\leq d_{\beta+1}$. By definition of $d_{\beta+1}$ as a bending, it is enough to show that $\rho_B\leq \omega\circ d$ and $\rho_B(x_n,y_n)\leq (\rho_n)_{cf}^{\beta+1}(x_n,y_n)$ for all $n\in\N$. The first inequality is direct, since $\rho_B\leq d_\beta\leq \omega\circ d$. For the second, fix $n\in \N$. By definition $\rho_B(x_n,y_n)\leq (\eta_n^\beta)_{cf}(x_n,y_n)$. Moreover, since $\eta_n^\beta$ is a bending of $(\rho_n)_{cf}^\beta$, we get that $\eta_n^\beta\leq (\rho_n)_{cf}^\beta$. Taking the curve-flat quotient of both pseudometrics we obtain $(\eta_n^\beta)_{cf}\leq (\rho_n)_{cf}^{\beta+1}$ and the desired inequality follows.

    Next we show that $d_{\beta+1}\leq \rho_B$. In this case we must prove that $d_{\beta+1}\leq d_\beta$ and $d_{\beta+1}(x_n,y_n)\leq (\eta_n^\beta)_{cf}(x_n,y_n)$ for every $n\in\N$. The inequality $d_{\beta+1}\leq d_\beta$ holds since both are bending pseudometrics of $\omega\circ d$, but the bending triplets of $d_{\beta+1}$ are smaller than those of $d_\beta$. Fix $n\in \N$. Since $\eta_n^\beta$ is the bending of $(\rho_n)_{cf}^\beta$ by a single pair $(x_n,y_n,d_\beta(x_n,y_n))$, it follows by Lemma \ref{Lemma:CF_of_pari_bending_formula} that 
    \begin{equation}(\eta_n^\beta)_{cf}(x_n,y_n)=\min\{d_\beta(x_n,y_n),(\rho_n)_{cf}^{\beta+1}(x_n,y_n)\}.\end{equation}
    As $d_{\beta+1}(x_n,y_n)\leq (\rho_n)_{cf}^{\beta+1}(x_n,y_n)$ and $d_{\beta+1}\leq d_\beta$, it follows that $d_{\beta+1}(x_n,y_n)\leq (\eta_n^\beta)_{cf}(x_n,y_n)$, as sought. 

    Now we fix $n\in\N$ and show that $\rho_{B_n}=\eta_n^{\beta+1}$. By Lemma \ref{Lemma:CF_of_pari_bending_formula}, the pseudometric $(\eta_n^\beta)_{cf}$ is equal to the bending of $(S_n,(\rho_n)_{cf}^{\beta+1})$ by the triplet 
    \begin{equation}(x_n,y_n,\min\{(\rho_n)_{cf}^{\beta+1}(x_n,y_n),d_\beta(x_n,y_n)\}),\end{equation}
    and $\rho_{B_n}$ is a further bending of this pseudometric by $\rho_B(x_n,y_n)$. Since we have already shown that $\rho_B=d_{\beta+1}$ and $d_{\beta+1}\leq d_\beta$, we obtain that $\rho_{B_n}$ is the bending of $(S_n,(\rho_n)_{cf}^{\beta+1})$ by the triplet
    \begin{equation}(x_n,y_n,\min\{(\rho_n)_{cf}^{\beta+1}(x_n,y_n),d_{\beta+1}(x_n,y_n)\}).\end{equation}
    Finally, observe that $d_{\beta+1}(x_n,y_n)\leq (\rho_n)_{cf}^{\beta+1}(x_n,y_n)$ by definition of $d_{\beta+1}$. In conclusion, we get that $\rho_{B_n}$ is the bending of $(S_n,(\rho_n)_{cf}^{\beta+1})$ by the triplet
     \begin{equation}(x_n,y_n,d_{\beta+1}(x_n,y_n)),\end{equation}
     and it it thus equal to $\eta_n^{\beta+1}$. This finishes the successor ordinal case.

    Let $\beta$ be a limit ordinal, and assume the claim holds for all $\gamma<\beta$. By definition, $\rho_{cf}^\beta =\inf_{\gamma<\beta} \rho_{cf}^\gamma$. For every $\gamma<\beta$, we know that $\rho_{cf}^\gamma$ is the attachment pseudometric in $Y$ defined by attaching the threads $(S_n,\eta_n^\gamma)$ to the frame $(M,d_\gamma)$, where 
    \begin{itemize}
        \item $(M,d_\gamma)$ is the bending of $(M,\omega\circ d)$ by the set $\{(x_n,y_n,(\rho_n)_{cf}^\gamma(x_n,y_n)\}_n$;
        \item $(S_n,\eta_n^\gamma)$ is the bending of $(S_n,(\rho_n)_{cf}^\gamma)$ by $\{(x_n,y_n,d_\gamma(x_n,y_n)\}$.
    \end{itemize}
    Denote by $d_\beta$ the pseudometric $\rho_{cf}^\beta$ restricted to $M$, which coincides with $\inf_{\gamma<\beta} d_\gamma$. Since $\inf_{\gamma<\beta} (\rho_n)_{cf}^\gamma=(\rho_n)_{cf}^\beta$ for every $n\in\N$, we obtain by the bending definition that $d_\beta$ equals the bending pseudometric obtained by bending $(M,\omega\circ d)$ by $\{x_n,y_n,(\rho_n)_{cf}^\beta(x_n,y_n)\}_n$. 

    Finally, fix $n\in\N$, and note that, denoting by $\eta_n^\beta$ the pseudometric $\rho_{cf}^\beta$ restricted to $S_n$, we have that $\eta_n^\beta=\inf_{\gamma<\beta} \eta_n^\gamma$. We apply again that $(\rho_n)_{cf}^\beta=\inf_{\gamma<\beta} (\rho_n)_{cf}^\gamma$, and that $d_\beta=\inf_{\gamma<\beta}d_\gamma$ to conclude that $\eta_n^\beta$ is the bending of $(S_n,(\rho_n)_{cf}^\beta)$ by the triplet $\{x_n,y_n,d_\beta(x_n,y_n)\}$. This proves the limit ordinal ordinal case.

\end{proof}

\textit{Proof of Theorem \ref{Theorem:main_theorem_2_general} continues.}

We apply the previous claim to $\beta=\alpha$ to prove (2) and (3). First we show that $M$ is total in $(Y,\rho_{cf}^\alpha)$. For all $z\in Y\setminus M$, there exists $n\in \N$ such that $z\in S_n$. Claim \ref{Claim2} implies that $\rho_{cf}^\alpha$ restricted to $S_n$ coincides with the bending $(\rho_n)_{cf}^\alpha$ by $\{(x_n,y_n,d_\alpha(x_n,y_n))\}$. In particular, $\rho_{cf}^\alpha$ in $S_n$ is smaller than $(\rho_n)_{cf}^\alpha$. Condition (S.4) of Claim \ref{Claim1} says that $\{x_n,y_n\}$ is total in $(S_n,(\rho_n)_{cf}^\alpha)$, and therefore $\{x_n,y_n\}$ is also total in $(S_n,\rho_{cf}^\alpha)$. As a consequence, $\rho_{cf}^\alpha(\{x_n,y_n\},z)=0$, as desired.

On the other hand, the pseudometric $\rho_{cf}^\alpha$ restricted to $M$ coincides with the bending pseudometric obtained bending $(M,\omega\circ d)$ by $\{(x_n,y_n,(\rho_n)_{cf}^\alpha(x_n,y_n)\}_n$. By equation \eqref{eq:S_k_collapse_by_ord_alpha}, we can apply Lemma \ref{Lemma:Full_bending_isom_better_small_scales} and it follows that for every $x\neq y\in M$ with $d(x,y)\leq 2^{-k}\diam(M,d)$ it holds that
\begin{equation}d(x,y)\leq \rho_{cf}^\alpha(x,y)\leq (1+2^{-k}\varepsilon)d(x,y).\end{equation}
In particular, this implies that
\begin{equation}d(x,y)\leq \rho_{cf}^\alpha(x,y)\leq (1+\varepsilon)d(x,y)\end{equation}
for all $x,y\in M$. Moreover, if $x,y\in M$ satisfy that $d(x,y)=\diam(M,d)$, then $(\omega\circ d)(x,y)=d(x,y)$ which finishes the proof of (2). 

Finally, we can directly apply Lemma \ref{Lemma:almost_isom_small_scales_geodesic} to prove that (3) holds for geodesic $M$. If, on the other hand, $M$ is a gapped segment, then $M$ can be written as $M=A\cup B$, where both $A$ and $B$ are geodesic and there exist two points $a\in A$ and $b\in B$ such that $d(x,y)=d(x,a)+d(a,b)+d(b,y)$ for all $x\in A$ and $y\in B$. Therefore, by Lemma \ref{Lemma:almost_isom_small_scales_geodesic} we have that $d=\rho^\alpha_{cf}$ when restricted to $A$ or $B$. Moreover, if we add the pair $(a,b)$ as $(x_1,y_1)$ in the admissible set $\Gamma$ for $(M,\alpha,\varepsilon)$, Claim \ref{Claim2} shows that $(x_1,y_1,(\rho_1)_{cf}^\alpha)$ is one of the bending triplets that defines $\rho_{cf}^\alpha$, which combined with (S.4) of Claim \ref{Claim1} proves that $d(a,b)=\rho_{cf}^\alpha(a,b)$. Combining both facts we obtain that $d(x,y)=\rho_{cf}^\alpha(x,y)$ for all $x,y\in M$, as desired. 

It only remains to show that $\alpha_{cf}(Y)=\alpha+\alpha_{cf}(M)$. However, this follows directly since $(M,d)$ is not purely $1$-unrectifiable and $Y/d_{cf}^\alpha$ is bi-Lipschitz equivalent to $(M,d)$. This finishes the proof for non purely $1$-unrectifiable metric spaces.

Let us assume that $M$ is purely 1-unrectifiable. Note that this also covers the case when $M$ is finite. Let $(\alpha_n)_n$ be an increasing sequence of ordinals less than $\alpha$ such that for every $\beta<\alpha$ there exists $N\in\N$ so that $\alpha_n\geq \beta$ for all $n>N$. Choose any point $p_0\in M$. For every $n\in\N$, we use the non purely $1$-unrectifiable case of the theorem to find a space $(S_n,\rho_n)$ such that $\diam(S_n,\rho_n)\leq 2^{-n}\diam(M)$ and $(S_n,(\rho_n)_{cf}^{\alpha_n})$ is isometric to the real line segment $[0,\frac{2^{-n}}{3}\diam(M,d)]$. Choosing any point in $S_n$, we may assume that $S_n\cap M=\{p_0\}$ for every $n\in\N$ and $S_n\cap S_m=\{p_0\}$ for every $n\neq m\in \N$. Now, consider the metric space $(Y,\rho)$ obtained by attaching $(S_n,\rho_n)$ to the frame $(M,d)$ (this attachment at a single point is sometimes called the $\ell_1$-sum of metric spaces). It is direct to prove that $(Y,\rho)$ is compact and that $\diam(Y,\rho)\leq 2\diam(M,d)$, so (1) holds. To prove that $(Y,\rho_{cf}^\alpha)$ is isometric to $(M,d)$, it suffices to note that for every $\beta\leq\alpha$, the space $(Y,\rho_{cf}^\beta)$ is the space obtained by attaching $(S_n,(\rho_n)_{cf}^\beta)$ to the frame $(M,d)$, and thus (2) and (3) follow, since $(S_n,(\rho_n)_{cf}^\alpha)$ is isometric to a single point metric space for every $n\in \N$. 

Finally, if $\beta<\alpha$, there exists $n\in\N$ such that $\beta\leq  \alpha_n$. Since $(Y,\rho_{cf}^{\alpha_n})$ contains the space $(S_n,(\rho_n)_{cf}^{\alpha_n})$, which is isometric to a non-trivial real line interval and not purely $1$-unrectifiable, we get that $\alpha_{cf}(Y)\geq\alpha_n\geq\beta$. We conclude that $\alpha_{cf}(Y)=\alpha=\alpha+\alpha_{cf}(M)$.

\end{proof}

\section{Open problems}
We propose some natural open questions. On one hand, it is natural to study if stronger non-universality results (in the sense of Theorem \ref{Main_Theorem_2_nocompactSUT}) can be obtained using the curve-flat index. This would be the case if, for instance, compact spaces enjoying a given bi-Lipschitz invariant property and with arbitrarily high countable curve-flat index were to be found. 

\begin{question}
    Let $P$ be any of the following bi-Lipschitz invariants:
    \begin{itemize}
        \item Finite Nagata dimension;
        \item Doubling Property (equivalently, finite Assouad dimension);
        \item Bi-Lipschitz embeddability into a finite dimensional euclidean space.
    \end{itemize}
    Does there exist a compact metric space universal for Lipschitz quotients onto compact spaces with property $P$? 
    Do there exist compact metric spaces of arbitrarily high countable curve-flat index satisfying property $P$?
\end{question}

We refer to \cite[Section 10.13]{Hei01} for the doubling property and the Assouad dimension, and to \cite{LanSch05} for the Nagata dimension.

We propose these three properties as they are natural and well studied bi-Lipschitz invariants that do not obviously prevent the curve-flat index to be higher than a given countable ordinal; and for which, to the best of our knowledge, it is unknown whether they admit a universal compact space for Lipschitz quotients. We believe that our construction in Theorem \ref{Theorem:main_theorem_2_general} yields spaces with finite Nagata dimension for \emph{finite} ordinals, provided we start with $M$ of finite Nagata dimension.  

On the other hand, some bi-Lipschitz invariants do force the curve-flat index to be lower than a given countable ordinal, but we still do not know whether they admit compact spaces universal for Lipschitz quotients. 

\begin{question}
    Let $P$ be any of the following bi-Lipschitz invariants:
    \begin{itemize}
        \item Pure $1$-unrectifiability;
        \item Curve-flat index smaller than $\alpha$, for some $\alpha<\omega_1$;
        \item Bi-Lipschitz embeddability into the real line.
    \end{itemize}
    
    Does there exist a compact metric space universal for Lipschitz quotients onto compact spaces with property $P$? 
\end{question}

Note that bi-Lipschitz embeddability into the real line implies that the curve-flat index is at most $1$. In fact, it follows from Proposition 6.6 and Theorem 6.5 in \cite{Flo+24} that this holds for all countably $1$-rectifiable spaces. This, in turn, raises the following question:

\begin{question}
    Given $n\in\N$, does there exist an ordinal $\alpha(n)<\omega_1$ such that $\alpha_{cf}(M)\leq \alpha(n)$ for all $M\subset \R^n$?
\end{question}

\section*{Acknowledgements} The authors are grateful to Yoël Perreau and Nikita Leo for helpful discussions on the topic. 

This work was supported by the Estonian Research Council grant (PRG1901). The second named author was also supported by Grant PID2021-122126NB-C33 funded by MICIU/AEI/10.13039/501100011033 and by ERDF/EU.
\printbibliography

\end{document}